\documentclass[11pt,a4paper,reqno]{amsart}
\usepackage{amsmath, amssymb, latexsym, enumerate,  graphicx, stmaryrd, eufrak,enumitem,
tikz}
\usepackage{ytableau}
\usepackage{thmtools}

\usepackage[pdftex]{hyperref}

\usepackage{bbm}
\usepackage{cases}
\usepackage{array}

\usepackage[all]{xy}

\usepackage[hmarginratio={1:1},vmarginratio={1:1},lmargin=80.0pt,tmargin=90.0pt]{geometry}

\usepackage{graphicx}
\usepackage{nicefrac}

\numberwithin{equation}{section}

\newtheorem{theorem}[subsubsection]{Theorem}
\newtheorem{lemma}[theorem]{Lemma}
\newtheorem{prop}[theorem]{Proposition}
\newtheorem{corollary}[subsubsection]{Corollary}

\newtheorem{MainTheorem}{Theorem}

\theoremstyle{definition}
\newtheorem{definition}[subsubsection]{Definition}

\newtheorem{remark}[theorem]{Remark}

\newtheorem{example}[subsubsection]{Example}

\newcommand{\indexi}{s}

\newcommand{\cB}{\mathcal{B}}

\newcommand{\hook}{\mathsf{h}}
\newcommand{\tR}{\mathsf{R}}

\newcommand{\Rep}{\mathsf{Rep}}

\newcommand{\tto}{\twoheadrightarrow}

\newcommand{\mN}{\mathbb{N}}
\newcommand{\mG}{\mathbb{G}}
\newcommand{\mZ}{\mathbb{Z}}
\newcommand{\mQ}{\mathbb{Q}}

\newcommand{\mC}{\mathbb{C}}

\newcommand{\End}{\mathrm{End}}
\newcommand{\im}{\mathrm{im}}

\newcommand{\Hom}{\mathrm{Hom}}

\newcommand{\0}{{\bar 0}}
\renewcommand{\1}{{\bar 1}}
\renewcommand{\k}{\Bbbk}
\newcommand{\sv}{\mathsf{v}}

\begin{document}
\title[Invariant theory]{Super invariant theory in positive characteristic}
\author{Kevin Coulembier, Pavel Etingof, Alexander Kleshchev and Victor Ostrik}
\address{K.C.: School of Mathematics and Statistics, University of Sydney, NSW 2006, Australia}
\email{kevin.coulembier@sydney.edu.au}

\address{P.E.: Department of Mathematics, MIT, Cambridge, MA USA 02139
}
\email{etingof@math.mit.edu}

\address{A.K.: Department of Mathematics,
University of Oregon, Eugene, OR 97403, USA}
\email{klesh@uoregon.edu}

\address{V.O.: Department of Mathematics,
University of Oregon, Eugene, OR 97403, USA}
\email{vostrik@math.uoregon.edu}

\dedicatory{To the memory of Irina Dmitrievna Suprunenko}
	

\thanks{
The work of the first author was partially supported by ARC grant DP200100712. 
The work of the second author was partially supported by the NSF grant DMS - 1916120.
The work of third author was partially supported by the NSF grant DMS-2101791 and 
Charles Simonyi Endowment at the Institute for Advanced Study. The third author also thanks the Isaac Newton Institute for Mathematical Sciences, Cambridge, for support and hospitality during the programme `Groups, representations and applications: new perspectives' where part of the work on this paper was undertaken. 
}


\begin{abstract}
We study invariant theory of the general linear supergroup in positive characteristic. In particular, we determine when the symmetric group algebra acts faithfully on tensor superspace and demonstrate that the symmetric group does not always generate all invariants.
\end{abstract}

\maketitle


\section*{Introduction}

For the entire paper we let $\k$ be a field of characteristic $p\ge 0$. We study invariant theory for the general linear supergroup over $\k$ when $p>2$.

If $p\not=2$, for $m,n\in\mN$, we consider the corresponding super vector space 
$$V=V_{m|n}:=\k^{m|n}.$$ For each $r\in\mZ_{>0}$, we have the algebra morphism
\begin{equation}\label{EqPhi}
\Phi=\Phi^{(r)}_{m|n}\,:\, \k S_r\;\to\; \End_{GL(m|n)}(V_{m|n}^{\otimes r}).
\end{equation}

Typically the `First Fundamental Theorem of Invariant Theory' establishes that $\Phi^{(r)}_{m|n}$ is surjective. This is known to be the case if $p=0$ or $n=0$, see for instance \cite{Berele, dCP}. 

The `Second Fundamental Theorem of Invariant Theory' then usually refers to descriptions of the kernel of $\Phi^{(r)}_{m|n}$; in coarsest form this is the answer to the question of when $\Phi^{(r)}_{m|n}$ is injective. Again the answer to the latter question is known and very succinct if $p=0$ or $n=0$, and does not depend on $p$ if $n=0$, see~\cite{Berele, dCP}.

In \cite{CEO} it was observed that the injectivity question for $\Phi^{(r)}_{m|n}$ {\em does} depend on $p$ when $mn\not=0$, and an upper bound on $r$ for which the morphism can be injective was given. In the current paper we show that this upper bound is in general not sharp, and determine the precise sharp bound. 
The question turns out to have a much more intricate answer than the previously known cases.

In fact, there exists a unique positive integer $r_p(m,n)$ such that the $\Phi^{(r)}_{m|n}$ is injective if and only if $r\leq r_p(m,n)$. Moreover, it is easy to see that $r_p(m,n)=r_p(n,m)$ so it is harmless to assume that $m\geq n$. We prove (see Theorem~\ref{TMain0}):

\begin{MainTheorem}\label{TA}
Let $p>2$, $m\geq n$ and set $d=(p+n-m)/2$. Then $$
r_p(m,n)\;=\; \begin{cases}
m +(p-1)n & \mbox{ if } d<2;\\
m+n+mn&\mbox{ if } d>n,\; (\mbox{\it i.e. }p>m+n);\\
\lceil  (n+1)p-1-d^2\rceil &\mbox{ if } 2\le  d\le n.\\
\end{cases}
$$
\end{MainTheorem}

We also demonstrate that the dependence on $p$ for the injectivity question implies that $\Phi^{(r)}_{m|n}$ is not always surjective either. 
We prove (see Theorem~\ref{ThmSurj}):

\begin{MainTheorem}\label{TB}
Let $p>2$.
\begin{enumerate}
\item[{\rm (i)}] If $r_p(m,n)\;< \; r\;\le \; m+n+mn$ then 
the morphism $\Phi_{m|n}^{(r)}$
is neither injective nor surjective.
\item[{\rm (ii)}] If the morphism $\Phi_{m|n}^{(r)}$
is surjective for all $r\in\mN$ then 
$mn=0$ or $m+n<p.$
\end{enumerate}
\end{MainTheorem}

It is an open question whether the necessary condition for surjectivity in Theorem~\ref{TB}(ii) is also sufficient. However, for $p=3$, we prove that it is (see Corollary~\ref{CNew}): 

\begin{MainTheorem}\label{TC}
For $p=3$, the morphisms $\Phi_{m|n}^{(r)}$
are surjective for all $r\in\mN$ if and only if $mn=0$ or $m=1=n$.
\end{MainTheorem}

We further give some estimates and examples for when the morphisms $\Phi^{(r)}_{m|n}$ are or are not surjective. For example, for $GL(1|1)$ the morphisms are proved to be always surjective, see Theorem~\ref{ThmSurj}(iii). On the other hand for $p=3$ we have that the morphisms $\Phi^{(2+3i)}_{2|1}$ 
are not surjective for all $i\in\mZ_{>0}$, see Theorem~\ref{Thm21}.

The above lack of surjectivity implies that to study invariant theory for the general linear supergroup in positive characteristic, one must look beyond the symmetric group. Our final result demonstrates how a a change of lattice in the integral symmetric group algebra has a promise to augment \eqref{EqPhi} to a surjective morphism. 
To be more precise, working over the integers, we define an intermediate $\mZ$-subalgebra $\Sigma_{m|n,\mZ}$ such that 
$$\im\, \Phi_{m|n,\mZ}^{(r)}\subseteq \Sigma_{m|n,\mZ}^{(r)}\subseteq \End(V^{\otimes r}_{m|n,\mZ}).$$
Informally speaking, the intermediate algebra $\Sigma_{m|n,\mZ}$ is obtained from 
$\im\, \Phi_{m|n,\mZ}^{(r)}$ by `inverting~$p$ where necessary', see Definition~\ref{DSigma} and the example in Theorem~\ref{TE}(iii) below. 
Then we define  $\Sigma_{m|n}^{(r)}:=\k\otimes\Sigma_{m|n,\mZ}^{(r)}$ and prove (see Proposition~\ref{PropSigma}):

\begin{MainTheorem}\label{TD}
Let $p>2$. Then we have:
\begin{enumerate}
\item[{\rm (i)}]  
$\dim_{\k}\Sigma_{m|n}^{(r)}\;=\; \dim_{\mC}\End_{GL(V_{m|n,\mC})}(V_{m|n,\mC}^{\otimes r})$.
\item[{\rm (ii)}]  There is a commutative diagram of algebra morphisms
$$\xymatrix{
\Sigma_{m|n}^{(r)}\ar@{^{(}->}[rr]^-{\widetilde{\Phi}^{(r)}_{m|n}}&& \End_{GL(m|n)}(V_{m|n}^{\otimes r})\\
\k S_r\ar[u]\ar[urr]_-{\Phi^{(r)}_{m|n}}}$$
where $\widetilde{\Phi}^{(r)}_{m|n}$ is injective.
\item[{\rm (iii)}]  The morphism $\widetilde{\Phi}^{(r)}_{m|n}$ is an isomorphism if and only if 
$$\dim_{\k}\End_{GL(m|n)}(V_{m|n}^{\otimes r})=\dim_{\mC}\End_{GL(m|n)}(V_{m|n,\mC}^{\otimes r}).$$
\end{enumerate}
\end{MainTheorem}

\vspace{3mm}
This suggests a potential approach to super invariant theory in positive characteristic:
\begin{enumerate}
\item Show that $\End_{GL(m|n)}(V_{m|n}^{\otimes r})$ does not depend on the characteristic $p$.
\item Determine explicitly the lattice $\Sigma_{m|n,\mZ}^{(r)}.$
\end{enumerate}

We demonstrate how this strategy works in the first case where the surjectivity of $\Phi^{(r)}_{m|n}$ fails (see Theorems~\ref{ThmSigma},\ref{Thm21}):

\begin{MainTheorem}\label{TE}
Let $p=3$, $r=5$, $m=2$ and $n=1$. 
\begin{enumerate}
\item[{\rm (i)}]  We have\, $\dim_\k \End_{GL(2|1)}(V_{2|1}^{\otimes 5})=120$ and\, $\dim\im \Phi^{(5)}_{2|1}=119$.
\item[{\rm (ii)}]  The morphism $\widetilde{\Phi}_{2|1}^{(5)}$ in Theorem~\ref{TD}(ii) is an isomorphism.
\item[{\rm (iii)}]  We have that $\Sigma^{(5)}_{2|1,\mZ}$ is the subring of $\mZ[\frac{1}{3}]S_5$ generated by $S_5$ and $\frac{1}{3}\sum_{\sigma\in S_5}\operatorname{sign}(\sigma)\sigma$.
\end{enumerate}
\end{MainTheorem}

The paper is organised as follows. In Section~\ref{SecPrel} we introduce some necessary background. In Section~\ref{SecInj} we determine completely when $\Phi$ is injective. In Section~\ref{SecSurj} we study when $\Phi$ is surjective, in particular establishing for which $m,n$ the morphisms~$\Phi$ are surjective for all $r$, when $p=3$. In Section~\ref{SecEx} we focus on the example $GL(2|1)$. In Section~\ref{SecFurther} we give a potential approach to a corrected `First Fundamental Theorem of Invariant Theory' for the general linear supergroup in positive characteristic.


\section{Preliminaries}\label{SecPrel}
We set $\mN=\{0,1,2,\cdots\}$.
\subsection{Faithful algebra modules}

Let $A$ be a finite dimensional algebra over $\k$. Label the isomorphism classes of simple $A$-modules by $L(i)$ and their projective covers by $P(i)$.

\begin{lemma}\label{LemDim}
Let $M$ be a finite dimensional $A$-module. 
\begin{enumerate}
\item[{\rm (i)}]  Assume that $\k$ is a splitting field for $A$.
 For each $i$, denote by $K(i)\subset P(i)$ the intersection of the kernel of all $A$-module morphisms $P(i)\to M$. Then the dimension of the image of\, $A\to \End_\k(M)$ is given by
$$\sum_i \dim_\k L(i) \dim_\k (P(i)/K(i)).$$
\item[{\rm (ii)}]  Assume that $A$ is self-injective, for instance $A=\k G$ for a finite group $G$. Then $M$ is faithful if and only if it contains every indecomposable projective as a direct summand. 
\end{enumerate}
\end{lemma}
\begin{proof}
For part (i), consider a decomposition of $1_A$ into orthogonal primitive (hence local) idempotents. It suffices to know the dimension of the image (so the codimension of the kernel) of $A\to\End_\k(M)$ restricted to $Ae\subset A$ for each idempotent $e$ in the decomposition.

 Fix such an idempotent $e$. It corresponds to some simple $L(i)$, and there are $\dim_\k L(i)$ idempotents in the decomposition corresponding to $L(i)$. The kernel of the first map below is clearly equal to that of the composite
 $$Ae\to\End_\k(M)\to\Hom_\k(eM,M)\simeq \Hom_\k(\Hom_A(Ae,M),M).$$
 This composite is just the evaluation homomorphism, so its kernel is the intersection of the kernels of all $A$-linear homomorphisms $Ae\to M$. The conclusion follows from the isomorphism $Ae\simeq P(i)$.
 
For part (ii), we start by observing that the algebra morphism $A\to\End_\k(M)$ can be interpreted as an $A$-module morphism $A\to M^{\dim M}$. If $M$ is faithful and the regular $A$-module injective, it follows that $A$ is a direct summand of $M^{\dim M}$. That $M$ contains every indecomposable projective module then follows from the Krull-Schmidt theorem. If $\k$ is a splitting field, the other direction follows from part (i). In general, we can observe that if $M$ contains every indecomposable projective, then $M^n$ contains $A$ for some $n\in\mN$, so clearly $M^n$ is faithful. But $A\to\End_{\k}(M^n)$ takes values in $\End_{\k}(M)^n$ so also $M$ is faithful.
\end{proof}

\subsection{Symmetric group}

For $r\in\mZ_{>0}$, we denote the symmetric group on $\{1,2,\cdots,r\}$ by~$S_r$. 

\subsubsection{}\label{part}

We give an index for notation we will use relating to compositions:
\begin{eqnarray*}
\Lambda(k)&& \mbox{Compositions of (maximal) length $k$}\\
\Lambda^+(k)&& \mbox{Partitions of (maximal) length $k$}\\
\Lambda(k,r)&& \mbox{Compositions of (maximal) length $k$ and of size $r$}\\
\Lambda^+(k,r)&& \mbox{Partitions of (maximal) length $k$ and of size $r$}\\
\Lambda(k|l)=\Lambda(k)\times \Lambda(l)&& \mbox{Bi-compositions of (maximal) lengths $k,l$}\\
\Lambda^+(k|l)=\Lambda^+(k)\times \Lambda^+(l)&& \mbox{Bi-partitions of (maximal) lengths $k,l$}\\
\Lambda(k|l,r)&& \mbox{Bi-compositions of (maximal) lengths $k,l$ and total size $r$}\\
\Lambda^+_p(k|l,r)&& \mbox{Elements $(\alpha,\beta)\in \Lambda^+(k|l)$ with $|\alpha|+p|\beta|=r$ }
\end{eqnarray*}
Here, given $\lambda=(\lambda_1,\dots,\lambda_k)\in\Lambda(k)$, we write $|\lambda|:=\lambda_1+\dots+\lambda_k$.

\subsubsection{}We denote the Specht module over $\k S_r$ corresponding to $\lambda\vdash r$ by $S^\lambda$. When we need to refer to the
Specht module over $\mC$ we will use the notation $S^\lambda_\mC$.

If $\mathrm{char}(\k)=p>0$, we denote the unique simple quotient of $S^\lambda$ by $D^\lambda$ for $p$-regular $\lambda$. These are non-isomorphic and exhaust all simple modules. We denote the projective cover of $D^\lambda$ by $P^\lambda$.

\subsubsection{} Let $(\lambda,\mu)\in \Lambda(k|l,r)$. Then we have the standard parabolic subgroup
$$
S_{\lambda}\times S_{\mu}:=(S_{\lambda_1}\times\dots\times S_{\lambda_k})\times(S_{\mu_1}\times\dots\times S_{\mu_l})\leq S_r.
$$
Let $\k_{\lambda|\mu}$ be the $1$-dimensional $\k(S_{\lambda}\times S_{\mu})$-module with $(g,h)\in S_{\lambda}\times S_{\mu}$ acting with $\operatorname{sign}(h)$. Define the sign-permutation module
$$
M^{\lambda|\mu}:=\operatorname{Ind}_{S_{\lambda}\times S_{\mu}}^{S_r}\k_{\lambda|\mu}.
$$

\subsubsection{Duality}
For any finite dimensional representation $M$ of a group, we have the canonical representation structure on $M^\ast=\Hom_{\k}(M,\k).$ We have
$$(M^{\lambda|\mu})^\ast\,\simeq\, M^{\lambda|\mu},$$
and also the simple modules are self-dual, by \cite[Theorem~11.5]{JArc}.

\subsection{The general linear supergroup} In this section we assume $p\not=2$.

\subsubsection{}For any super vector space $V$ over $\k$, we have the associated affine group superscheme $GL(V)$, with associated Lie superalgebra $\mathfrak{gl}(V)$. The latter is the space of all $\k$-linear morphisms $V\to V$ equipped with super commutator.

\subsubsection{}\label{basis}We consider the superspace $V_{m|n}=\k^{m|n}$.
When $\k$ is clear from context, we also write $GL(m|n):=GL(V_{m|n})$ and $\mathfrak{gl}(m|n):=\mathfrak{gl}(V_{m|n})$.

Let $\{\sv_1,\dots,\sv_m\}$ be the standard basis of $(V_{m|n})_\0$ and $\{\sv_1',\dots,\sv_n'\}$ be the standard basis of $(V_{m|n})_\1$. We have a corresponding basis 
$$\{E_{i,j},E_{a',b'},E_{a',j},E_{i,b'}\mid 1\leq i,j\leq m, 1\le a,b \le n\}$$
of $\mathfrak{gl}(V_{m|n})$.
For example $E_{a',j}$ is defined via $E_{a',j}\sv_j=\sv_a'$, while all other standard basis elements are being sent to zero.

\subsubsection{}\label{weights} We fix the basis $\{\varepsilon_i,\varepsilon'_j\,|\, 1\le i\le m, 1\le j\le n\}$ for $\mathfrak{h}^\ast$, with $\mathfrak{h}\subset\mathfrak{gl}(m|n)$ the standard Cartan subalgebra of diagronal matrices, as follows. For all $1\le i\le m, 1\le j\le n$
$$\varepsilon_i(E_{j,j})=\delta_{ij}\quad\mbox{and}\quad \varepsilon_i(E_{j',j'})=0$$
$$\varepsilon'_i(E_{j,j})=0\quad\mbox{and}\quad \varepsilon'_i(E_{j',j'})=\delta_{ij}.$$

We use the same symbols for the corresponding basis of the character group $X(T)\simeq \mZ^{m+n}$ of the maximal torus $T$ of diagonal matrices in $GL(m)\times GL(n)$.
With this fixed basis we can interpret elements of $\Lambda(m|n)$, or more generally of $\mZ^{m+n}$ as characters (weights). More explicitly, we set
$$\Lambda(m|n)\hookrightarrow X(T),\quad (\alpha,\beta)\mapsto (\alpha|\beta):=\sum_i \alpha_i\varepsilon_i+\sum_j\beta_j\varepsilon'_j.$$

\subsection{Tensor superspace}

\subsubsection{}One of the many equivalent formulations of invariant theory for the general linear supergroup asks for a description of the algebra of $GL(m|n)$-invariant endomorphisms of tensor superspace:
$$\End_{GL(m|n)}(V_{m|n}^{\otimes r});$$
typically in terms of the symmetric group.

\subsubsection{}

For a positive integer $r$, there is a structure of a left $\k S_r$-module on the tensor superspace $V_{m|n}^{\otimes r}$ such that for all homogeneous $v_1,\dots,v_r\in V_{m|n}$ and $\sigma\in S_d$, we have 
\begin{equation*}\label{ESiAct*}
	\sigma(v_1\otimes\dots\otimes v_{r})=
	(-1)^{\langle\sigma;v_1,\dots,v_r\rangle} v_{\sigma^{-1}1}\otimes\dots\otimes v_{\sigma^{-1} r},
\end{equation*}
where 
\begin{equation*}\label{EAngleSi}
	\langle\sigma;v_1,\dots,v_r\rangle:=\sharp\{1\leq k<l\leq r\mid  \sigma k>\sigma l\ \text{and}\ v_k,v_l\ \text{are odd}\}.
\end{equation*}
This defines the algebra morphism
\begin{equation}\label{EqPhi0}
\k S_r\;\to\;\End_{\k}(V_{m|n}^{\otimes r}),
\end{equation}
where the right-hand side consists of all homogeneous endomorphisms of the superspace in the argument. By definition of $GL(m|n)$, this morphism actually takes values in $GL(m|n)$-endomorphisms, leading to~\eqref{EqPhi}.

We will occasionally also use the complex and integral versions $V_{m|n,\mC}=\mC^{m|n}$ and $V_{m|n,\mZ}=\mZ^{m|n}$, so that the $\k S_r$-module  $V_{m|n}^{\otimes r}$ is reduction modulo $p$ of $V_{m|n,\mC}^{\otimes r}$ using the lattice $V_{m|n,\mZ}^{\otimes r}$.


\section{Injectivity}\label{SecInj}
In this section we assume that $p\not=2$ unless stated otherwise.

\subsection{Main result and examples}
\label{SSTS}

It is clear that there exists a unique positive integer $r_p(m,n)$ such that the $\k S_r$-module $V_{m|n}^{\otimes r}$ is faithful if and only if $r\leq r_p(m,n)$. In this section we determine this value. First, we give an overview of what is known.

\begin{remark}\label{RP}

\begin{enumerate}[label=(\roman*)]
\item In characteristic zero, we have $r_0(m,n)=m+n+mn$. Indeed, by \cite[3.20]{Berele}, the $\mC S_r$-module $V_{m|n,\mC}^{\otimes r}$ contains all irreducible $\mC S_r$-modules as summands if and only if $r\leq m+n+mn$. Now, we can apply Lemma~\ref{LemDim}(ii). 
\item We have $r_p(m,n)=r_p(n,m)$. This is follows from the fact that the $\k S_r$-module $V_{m|n}^{\otimes r}$ is obtained from the $\k S_r$-module $V_{n|m}^{\otimes r}$ by tensoring with the sign representation. 
\item We have $r_p(m,0)=m$. This is well-known. For example by  Lemma~\ref{LUpper} below, we have $r_p(m,0)\leq m$, and for the opposite inequality it suffices to note that the regular representation is a summand of $V_{m|0}^{\otimes m}$. 
\item The question makes sense for $p=2$ in which case, by (iii), we find $r_2(m,n)=m+n$.
\item In \cite[Example~4.3]{CEO}, the following lemma was observed, which demonstrates dependence of $r_p(m,n)$ on $p$, contrary to the value $r_p(m,0)$. 
\end{enumerate}

\end{remark}

\begin{lemma}\label{LUpper}
If $p>2$, then $r_p(m,n)\leq m+n+\min(mn,(p-2)m,(p-2)n).$
\end{lemma}
\begin{proof}
In view of Remark~\ref{RP}(ii), we may assume that $m\geq n$ and then prove $r_p(m,n)\leq m+n+n\min(m,p-2).$
Suppose first that $m< p-2$. We have to prove that for $r:=m+n+mn+1$ the $\k S_r$-module $V_{m|n}^{\otimes r}$ is not faithful. Let $\lambda$ be the partition $((n+1)^{m+1})$ of $r$. We denote the corresponding row centralizer by $R_\lambda$ and the corresponding column centralizer by $C_\lambda$. Let $a_\lambda=\sum_{\sigma\in R_\lambda}\sigma$ and $b_\lambda=\sum_{\sigma\in C_\lambda}\operatorname{sign}(\sigma)\sigma$. These are considered as elements of the group algebra $R S_n$ for $R=\k,\mC$ or $\mZ$ depending on the context. Note that $a_\lambda b_\lambda$ is a multiple of Young's idempotent corresponding to $\lambda$, so it annihilates all irreducible $\mC S_r$-modules except $S^\lambda_\mC$. Moreover, $a_\lambda b_\lambda$ is non-zero when considered as an element of $\k S_r$. On the other hand, by \cite[3.20]{Berele}, $V_{m|n,\mC}^{\otimes r}$ is missing the irreducible constituent $S^\lambda_\mC$, so $a_\lambda b_\lambda V_{m|n,\mC}^{\otimes r}=0$, hence $a_\lambda b_\lambda V_{m|n,\mZ}^{\otimes r}=0$, hence $a_\lambda b_\lambda V_{m|n}^{\otimes r}=0$. 

Suppose $m\geq p-2$. We have to prove that for $r:=m+(p-1)n+1$ the $\k S_r$-module $V_{m|n}^{\otimes r}$ is not faithful. We have a basis $u_1\otimes \dots\otimes u_r$ of $V_{m|n}^{\otimes r}$, where each $u_i$ is equal to some standard basis element $\sv_s$ or $\sv_t'$. Since $r=m+(p-1)n+1$, either $u_i=u_j=\sv_s$ for some $s$ and some distinct $i,j$, or $u_{i_1}=\dots=u_{i_p}=\sv_t'$ for some $t$ and some distinct $i_1,\dots,i_p$. In both cases, it follows easily that  $(\sum_{\sigma\in S_r}\operatorname{sign}(\sigma)\sigma)(u_1\otimes \dots\otimes u_r)=0$. 
\end{proof}

In view of Remark~\ref{RP}(ii), we may assume that $m\geq n$ whenever convenient. We also henceforth focus only on $p>2$. The goal of this section is to prove the following theorem:

\begin{theorem}\label{TMain0} 
Let $p>2$, $m\geq n$ and set $d=(p+n-m)/2$. Then the maximal value of $r$ for which $\Phi^{(r)}_{m|n}$ in \eqref{EqPhi} is injective is
$$
r_p(m,n)\;=\; \begin{cases}
m +(p-1)n & \mbox{ if } d<2;\\
m+n+mn&\mbox{ if } d>n,\; (\mbox{\it i.e. }p>m+n);\\
\lceil  (n+1)p-1-d^2\rceil &\mbox{ if } 2\le  d\le n.\\
\end{cases}
$$
\end{theorem}

In the following subsections we will actually prove an alternative expression. For this, we introduce, for $s\in\mN$:
\begin{equation}\label{Es}
t^\indexi_p(m,n):=\indexi(m+1)+(p-\indexi)(n+1-\indexi)-1 = \indexi m+(p-\indexi)n+ \indexi(\indexi-p)+p-1
\end{equation}
and 
\begin{equation}\label{ER2}
t=t_p(m,n):=\min\{t^\indexi_p(m,n)\mid 1\leq \indexi\leq \min(p/2,n+1)\}.
\end{equation}

\begin{theorem}\label{TMain} 
Let $p>2$ and 
$m\geq n$. Then $
r_p(m,n)=t_p(m,n).$
\end{theorem}


If $n=0$, both theorems return $r_p(m,0)=m$. So  
in view of Remark~\ref{RP}(iii), we may assume that $n\neq 0$, so from now on we assume that $m\geq n\geq 1$.

\subsubsection{Equivalence of Theorems~\ref{TMain0} and~\ref{TMain}}
Firstly we can observe that 
\begin{equation}\label{ERange}
t_p(m,n)=\min\{t^\indexi_p(m,n)\mid 1\leq \indexi\leq n+1\}.
\end{equation}
Indeed, for \eqref{Es}, interpreted as a (quadratic) function of $s:\mathbb{R}\to\mathbb{R}$, the unique local minimum is achieved at
$d=(p+n-m)/2\leq p/2.$
If $d$ is an integer this is the unique minimum of $s:\mZ\to \mZ$. If $d$ is a half integer, then the minimum value is achieved at $d\pm 1/2$. This completes the proof of (\ref{ERange}).

Now, the case $d<2$ corresponds to the case where the minimum of $s:\mZ\to \mZ$ is achieved either left of the interval $[1,n+1]$ or at $s=1$. Both cases lead to $t_p(m,n)=t^1_p(m,n)=m+(p-1)n$. 
The case $d>n$ leads to $t_p(m,n)=t^{n+1}_p(m,n)=m+n+mn$.
The case $2\leq d\leq n$ leads to $t_p(m,n)=t^{d}_p(m,n)$ or $t_p(m,n)=t^{d\pm1/2}_p(m,n)$ both of which yield $\lceil  (n+1)p-1-d^2\rceil$.

\label{RemMain}

\begin{remark} \label{R2}
\begin{enumerate}[label=(\roman*)]
\item It is easy to see that $t\geq m$.

%

\item We can restate Lemma~\ref{LUpper} as the claim $r_p(m,n)\le \min(t^1_p(m,n),t^{n+1}_p(m,n))$.
\end{enumerate}
\end{remark}

\begin{example}
We continue to assume $m\ge n\ge 1$.
\begin{enumerate}[label=(\roman*)]
\item If $n=1$,  then
$$r_p(m,1)=m+1+\min(m,p-2),$$ which is exactly the upper bound of  Lemma~\ref{LUpper}.  

\item If $p=3$, then 
$$r_3(m,n)=m+2n,$$ which is exactly the upper bound of  Lemma~\ref{LUpper}.  

\item If $p=5$, then 
$$r_5(m,n)=\min(m+4n,2m+3n-2),$$
which can be below the upper bound $\min(m+4n,m+n+mn)$ of Lemma~\ref{LUpper}, see Example~\ref{EEx1}.

\item If $p=7$, then
$$r_7(m,n)=\min(m+6n,2m+5n-4,3m+4n-6),$$
where the last term $3m+4n-6$ has to be omitted if $n=1$.

\item For a fixed $p$, the generic value of $r_p(m,n)$, i.e. for $m\gg n$, is $m+(p-1)n$. As has to be the case a priori (see \cite[Remark~4.4]{CEO}), for fixed $m,n$ and generic $p\gg 0$, we have $r_p(m,n)=r_0(m,n)$.
\end{enumerate}
\end{example}

\begin{example}\label{EEx1} 
Let $p=5$ and $m=n=7$. We have $t_5^1(7,7)=35$ and $t_5^2(7,7)=33$, so $r_5(7,7)=33$. This is below the upper bound $$\min(t_5^1(7,7), t_5^8(7,7))=\min(35,63)=35$$
from Lemma~\ref{LUpper}.
%
%
\end{example}

We will need the following observation:

\begin{lemma}\label{L<p}
If $\alpha=(\alpha_1,\dots,\alpha_k)$ is a partition of $t$ then $\alpha_{m+1}\leq \min(p-1,n)$. 
\end{lemma}
\begin{proof}
We will use the description of $t$ given in (\ref{ERange}). 

Suppose $\alpha_{m+1}>p-1$. Then $\alpha_{m+1}\geq p$ and so 
$$(m+1)p\leq \alpha_1+\dots+\alpha_k=t\leq t^i_p(m,n)$$ for all $1\leq i\leq n+1$. If $p\leq n+1$, then we have $t^p_p(m,n)=p(m+1)-1$ giving a contradiction. If $p> n+1$ then 
$$(m+1)p>(m+1)(n+1)-1=t^{n+1}_p(m,n),$$ giving a contradiction again. 

On the other hand, suppose $\alpha_{m+1}>n$. Then $\alpha_{m+1}\geq n+1$ and so 
$$(m+1)(n+1)\leq \alpha_1+\dots+\alpha_k =t\leq t^{n+1}_p(m,n)
=(n+1)(m+1)-1,
$$
giving a contradiction. 
\end{proof}

\subsection{Indecomposable summands of tensor superspace}\label{SecDon}
Recall the integer $t=t_p(m,n)$ from (\ref{ER2}) or (\ref{ERange}). 
To prove Theorem~\ref{TMain}, it suffices to show that the $\k S_t$-module $V_{m|n}^{\otimes t}$ is faithful, while the $\k S_{t+1}$-module $V_{m|n}^{\otimes (t+1)}$ is not. By Lemma~\ref{LemDim}(ii), this is equivalent to the statement that $V_{m|n}^{\otimes t}$ contains every projective indecomposable $\k S_t$-module as a summand, while $V_{m|n}^{\otimes (t+1)}$ is missing some projective indecomposable $\k S_{t+1}$-module as a summand. We therefore discuss these indecomposable summands, 
 following~\cite{Donkin}.
 
 Unless further specified, we consider arbitrary $k,l,r\in\mN$.

\subsubsection{}
Recall the notation for compositions and partitions from~\ref{part}. For $1\leq i\leq k$, we denote
$$
\epsilon_i:=(0,\dots,0,1,0,\dots,0)\in\Lambda(k)
$$
with $1$ in the $i$th position.  We have the usual dominance order on $\Lambda(k)$ and on each $\Lambda(k,r)$.

We will identify $\Lambda(k|l,r)$ with $\Lambda(k+l,r)$ via the bijection mapping $(\lambda,\mu)\in \Lambda(k|l,r)$ to $(\lambda_1,\dots,\lambda_k,\mu_1,\dots,\mu_l)\in\Lambda(k+l,r)$; the dominance order on $\Lambda(k|l,r)$ is then inherited from the dominance order on $\Lambda(k+l,r)$ via this identifcation.

\subsubsection{}Given $(\lambda,\mu)\in \Lambda(k|l,r)$, we denote
by $(V_{k|l}^{\otimes r})_{\lambda|\mu}$ the span of all $u_1\otimes \dots\otimes u_r$ such that for each $1\leq i\leq k$, exactly $\lambda_i$ of the $u_k$'s are equal to $\sv_i$, and for each $1\leq j\leq l$, exactly $\mu_j$ of the $u_k$'s are equal to $\sv_j'$, where $\{\sv_i,\sv_j'\}$ is our standard basis from \ref{basis}. We have a $\k S_r$-module decomposition
\begin{equation}\label{ETensWts}
V_{k|l}^{\otimes r}\simeq \bigoplus_{(\lambda,\mu)\in\Lambda(k|l,r)}(V_{k|l}^{\otimes r})_{\lambda|\mu}.
\end{equation}
Moreover, it is clear that for every $(\lambda,\mu)\in\Lambda(k|l,r)$, we have that 
\begin{equation}\label{EVMIso}
(V_{k|l}^{\otimes r})_{\lambda|\mu}\simeq M^{\lambda|\mu}
\end{equation}
 as $\k S_r$-modules. This explains our interest in sign-permutation modules. 
The indecomposable summands of sign-permutation modules have been studied in \cite{Donkin}. 

\subsubsection{} Donkin \cite{Donkin} defines the {\em Schur superalgebra} $S(k|l,r)$. 
By \cite[2.3(1)]{Donkin} (see also \cite[Theorem 5.2]{BK}), we have 
$$S(k|l,r)\simeq \End_{\k S_r}(V_{k|l}^{\otimes r}).$$ 
The algebra $S(k|l,r)$ comes with the orthogonal idempotents $\{\xi_{\lambda|\mu}\mid (\lambda,\mu)\in\Lambda(k|l,r)\}$ summing to the identity. Under the isomorphism $S(k|l,r)\simeq \End_{\k S_r}(V_{k|l}^{\otimes r})$, the idempotent $\xi_{\lambda|\mu}$ is the projection onto the summand $(V_{k|l}^{\otimes r})_{\lambda|\mu}$ as in (\ref{ETensWts}). For any $S(k|l,r)$-module $M$ and $(\lambda,\mu)\in\Lambda(k|l,r)$, the $(\lambda,\mu)$-weight space of $M$ is defined as 
$$
M_{\lambda|\mu}:=\xi_{\lambda|\mu}M.
$$

We will often need to work under the additional assumption $r\leq k,l$.
Recall
$$
\Lambda^+_p(k|l,r):=\{(\alpha,\beta)\in\Lambda^+(k)\times\Lambda^+(l)\mid |\alpha|+p|\beta|=r\}.
$$
Note that for $(\alpha,\beta)\in\Lambda_p(k|l,r)$, we have $(\alpha,p\beta)\in\Lambda(k|l,r)$, where $p\beta:=(p\beta_1,\dots,p\beta_m)$.

\begin{lemma} {\rm \cite[2.3(4)]{Donkin}, \cite[Theorem 5.5]{BK}}
Let $r\leq k,l$. 
For each $(\alpha,\beta)\in\Lambda_p^+(k|l,r)$ there exists an irreducible $S(k|l,r)$-module $L(\alpha|p\beta)$ with highest weight $(\alpha,p\beta)$, and 
$$\{L(\alpha|p\beta)\mid (\alpha,\beta)\in\Lambda_p^+(k|l,r)\}$$ is a complete irredundant set of irreducible $S(k|l,r)$-modules. 
\end{lemma}

\subsubsection{}In case $r\leq k,l$, the natural functor ${\mathcal F}_{k|l,r}=\Hom_{S(k|l,r)}(V^{\otimes r},-)$ from $S(k|l,r)$-modules to $\k S_r$-modules is exact. 
From now on it will be sufficient to work in the case $k=l=r$.  
Denote by $P(\alpha|p\beta)$ the projective cover of the $S(r|r,r)$-module $L(\alpha|p\beta)$ for $(\alpha,\beta)\in\Lambda_p^+(r|r,r)$. Then we have the $\k S_r$-modules 
$$
Y^{\alpha|p\beta}:={\mathcal F}_{r|r,r}(P(\alpha|p\beta))
$$
for all $(\alpha,\beta)\in\Lambda_p^+(r|r,r)$, called {\em sign-Young modules}. 

\subsubsection{} Now we consider again $m\ge n\ge 1$ and $t=t_p(m,n)$ from~\eqref{ER2}. We are thus interested in $V_{m|n}^{\otimes r}$ for $r=t$ or $r=t+1$. Note by Remark~\ref{R2}(i) that in both cases $n\leq m\leq r$. So every composition with $m$ parts can be considered as a composition with $r$ parts by adding $r-m$ zeros to the end, i.e. we naturally identify $\Lambda(m)$ with a subset of $\Lambda(r)$, and similarly for $\Lambda(n)$. In this way we will always consider elements of $\Lambda(m|n,r)$ as elements of $\Lambda(r|r,r)$. So the  weight space $L(\alpha|p\beta)_{\lambda|\mu}$ in the following theorem makes sense. Note also that the $\k S_r$-module $
M^{\lambda|\mu}$ is the same whether we consider $(\lambda,\mu)$ as an element of $\Lambda(m|n,r)$ or $\Lambda(r|r,r)$. 

\begin{theorem}\label{TDKlyachko} {\rm \cite[2.3(6),(7)]{Donkin}}
The modules 
$$\{Y^{\alpha|p\beta}\mid (\alpha,\beta)\in\Lambda_p^+(r|r,r)\}$$  are pairwise non-isomorphic and indecomposable. Moreover, for any $(\lambda,\mu)\in\Lambda(m|n,r)$, we have 
$$
M^{\lambda|\mu}\simeq\bigoplus_{(\alpha,\beta)\in\Lambda_p^+(m|n,r)}(\dim L(\alpha|p\beta)_{\lambda|\mu})Y^{\alpha|p\beta},
$$
for $L(\alpha,p\beta)$ the simple $S(r|r,r)$-module.
\end{theorem}

Recall that a partition $\alpha=(\alpha_1,\dots,\alpha_k)$ is called $p$-restricted if $\alpha_i-\alpha_{i+1}<p$ for all $i=1,\dots,k-1$ and $\alpha_k<p$. 
By the classical theory of Young modules developed in \cite{JArc,Klyachko}, for a $p$-restricted partition $\alpha$ of $r$, we have that $Y^{\alpha|0}$ is the projective cover of the irreducible module $D^{\alpha'}\otimes\operatorname{sign}$. 
So, taking into account Lemma~\ref{LemDim}(ii), the $\k S_r$-module $V_{m|n}^{\otimes r}$ is faithful if and only if it contains every $Y^{\alpha|0}$ with $p$-restricted $\alpha\in\Lambda^+(r,r)$ as a summand. In view of (\ref{ETensWts}), (\ref{EVMIso}) and Theorem~\ref{TDKlyachko}, to prove Theorem~\ref{TMain}, it suffices to prove the following

\begin{prop}\label{PMainOld} 
Let $m\geq n\geq 1$ and $t$ be as in (\ref{ER2}). 
\begin{enumerate}
\item[{\rm (i)}] Then for every $\alpha\in\Lambda^+(t)$ there exists $(\lambda,\mu)\in\Lambda(m|n,t)$ such that $L(\alpha|0)_{\lambda|\mu}\neq 0$.
\item[{\rm (ii)}] There exists a $p$-restricted partition $\alpha\in\Lambda^+(t+1)$ such that $L(\alpha|0)_{\lambda|\mu}= 0$ for all $(\lambda,\mu)\in\Lambda(m|n,t+1)$.
\end{enumerate}
\end{prop}

Note that Theorem~\ref{TMain} would follow from the weaker statement than Proposition~\ref{PMainOld}, namely, we could assume in \ref{PMainOld}(i) that $\alpha$ is $p$-restricted, but this assumption happens to be unnecessary. 

The proof of the first part of this proposition will be given in \S\ref{SSMainProof} and the proof of the second part will be given in \S\ref{SSKevinCounterExamples}.

\begin{remark} 
In \cite{Sup}, Irina Suprunenko proved that for the irreducible $GL(n)$-module $L(\lambda)$ with {\em $p$-restricted}\, highest weight $\lambda$ and any weight $\mu$, 
the weight space $L(\lambda)_\mu$ is non-trivial in characteristic $p$ if and only if it is non-trivial in characteristic $0$. In other words, while the weight multiplicities can be much smaller in characteristic $p$ than in characteristic $0$, weights do not `disappear', as longs as the highest weight is $p$-restricted. (A similar result for other reductive algebraic groups was proved by Premet \cite{Pr}). 

Proposition~\ref{PMainOld} can be considered as a development of the work of Irina Suprunenko to the case of supergroups. Part (ii) of the proposition shows that the statement analogous to that of Suprunenko's theorem in general fails for $GL(m|n)$, while part (i) of the proposition shows that it holds in a certain range. Our proof of part (i) using lowering operators is inspired by Suprunenko's techniques.
\end{remark}

\subsection{Lower bound}
\label{SSMainProof}
In this section we prove Proposition~\ref{PMainOld}(i). Let $\alpha\in\Lambda^+(t)$. We set $L:=L(\alpha|0)$ and denote by $v_+$ a highest weight vector of $L$. We will act on $v_+$ with appropriate `lowering operators' to get a non-zero vector in the required weight space $L_{\lambda|\mu}$. 

\subsubsection{}\label{SchurGL} As explained in \cite[p.25]{BK}, the category of supermodules over the Schur algebra $S(t|t,t)$ is isomorphic to the category of degree $t$ polynomial representations of the supergroup $GL(t|t)$, which in turn can be considered as a subcategory of the  category of integrable representations over the distribution superalgebra $\operatorname{Dist}(GL(t|t))$, see \cite[Corollary 3.5]{BK}. Under this identification, the weight spaces as defined in Section~\ref{SecDon} correspond to the $GL(t|t)$-weight spaces, via the embedding in~\ref{weights}.

 Moreover, by \cite[Theorem 3.2]{BK}, the distribution superalgebra $\operatorname{Dist}(GL(t|t))$ is isomorphic to the hyperalgebra $U_\k$, which is obtained by extending scalars from $\mZ$ to $\k$ in the Kostant $\mZ$-form $U(\mathfrak{gl}(t|t,\mC))_\mZ$ of the universal enveloping superalgebra $U(\mathfrak{gl}(t|t,\mC))$. 

The explicit description of the hyperalgebra $U_\k$ is given in \cite[\S3]{BK}---from it we will only need the basis elements of $\mathfrak{gl}(t|t)\subset U_{\k}$
$$\{E_{i,j},E_{i',j'},E_{i',j},E_{i,j'}\mid 1\leq i,j\leq t\}$$
from~\ref{basis}, which satisfy the commutation relations \cite[(3.2)]{BK}. It will be especially important for us to use the commuting relation 
\begin{equation}\label{ECommE}
E_{j,i'}E_{i',j}=E_{j,j}+E_{i',i'}-E_{i',j}E_{j,i'}
\end{equation}
(recall that $E_{i',j},E_{j,i'}$ are odd), and the fact that for all $(\beta,\gamma)\in\Lambda(t|t,t)$ we have 
\begin{equation}\label{EWtAction}
E_{i',j}L_{\beta|\gamma}\subseteq L_{\beta-\epsilon_j|\gamma+\epsilon_i}.
\end{equation} 
Finally, since $E_{i',j}$ is odd and $[E_{i',j},E_{i',j}]=0$, we have in $U_\k$:
\begin{equation}\label{EOddSquare}
E_{i',j}^2=0.
\end{equation}

\begin{lemma} \label{L'} We have  $E_{i',j'}v_+=0$ for all $1\leq i,j\leq t$.
\end{lemma}
\begin{proof}
 If $i'=j'$ then $E_{i',i'}v_+=0$ since $v_+$ has weight $(\alpha|0)$. Otherwise, since the weight $(\alpha|\epsilon_i-\epsilon_j)$ does not appear in polynomial representations (as it does not appear in tensor powers of $V$), so it follows that $E_{i',j'}v_+=0$.
\end{proof}

Let $\alpha$ have $h$ non-zero parts, i.e. $\alpha=(\alpha_1,\dots,\alpha_h)$ with $\alpha_1\geq\dots\geq\alpha_h>0$. If $h\leq m$ we can take $(\lambda,\mu)=(\alpha|0)$ in Proposition~\ref{PMainOld}(i). So we assume from now on that $h>m$. We identify the partition $\alpha$ with its Young diagram, which consists of the {\em boxes} $(a,b)\in \mZ\times \mZ$ satisfying $1\leq a\leq h$ and $1\leq b\leq \alpha_a$. 

Given integers $c,d$ such that $m<c\leq d\leq h$, 
we consider the subset 
$$
\alpha_{[c,d]}:=\{(a,b)\in\alpha\mid c\leq a\leq d\}\subset \alpha
$$
of the boxes of the Young diagram $\alpha$ in the rows between $c$ and $d$. 
Fix an integer $e$ such that $\alpha_c\le e\le t$, or in other words
\begin{equation}\label{EAssumption}
\{e-\alpha_c+1,e-\alpha_c+2,\dots,e\}\subseteq \{1,\dots,t\}.
\end{equation} 
For every box $(a,b)\in \alpha_{[c,d]}$ we set 
$$i_{e;a,b}:=e-\alpha_a+b.$$
We note that 
$$1\le i_{e;a,b}\le e\le t$$ for all $(a,b)\in \alpha_{[c,d]}$.

For each integer $a$ satisfying $c\leq a\leq d$, we 
consider the following elements of $U_\k$ and $L$:
\begin{eqnarray*}
Y^{(e)}_{a}&:=&E_{i_{e;a,1}',a}E_{i_{e;a,2}',a}\cdots  E_{i_{e;a,\alpha_a}',a}
\\
X^{(e)}_{a}&:=&E_{a,i_{e;a,\alpha_a}'}E_{a,i_{e;a,\alpha_a-1}'}\cdots E_{a,i_{e;a,1}'},
\end{eqnarray*}
and 
\begin{eqnarray}\label{EXYLoc1}
Y^{(e)}_{[a,d]}&:=&Y^{(e)}_aY^{(e)}_{a+1}\cdots Y^{(e)}_d,
\\
X^{(e)}_{[a,d]}&:=&X^{(e)}_dX^{(e)}_{d-1}\cdots X^{(e)}_a,
\label{EXYLoc2}
\\ 
v_a&:=& Y^{(e)}_{[a,d]}v_+.\label{EXYLoc3}
\end{eqnarray}
By convention, we set $v_{d+1}:=v_+.$
Note that the distinct (odd) elements $E_{i_{e;a,b}',a}$ appearing in the product $Y^{(e)}_{[c,d]}$ (resp. $X^{(e)}_{[c,d]}$) anti-commute, so without specifying the order of the product, we can simply write  
$$
Y^{(e)}_{[c,d]}=\pm\prod_{(a,b)\in \alpha_{[c,d]}}E_{i_{e;a,b}',a},\;\mbox{and }\;\;
X^{(e)}_{[c,d]}=\pm\prod_{(a,b)\in \alpha_{[c,d]}}E_{a,i_{e;a,b}'}.
$$

\begin{example}\label{Ex2}
Consider the case $p=5$, $m=n=7$ and $\alpha=(3^{10},1^3)$, so $h=13$. By Example~\ref{EEx1}, we have $t=33$, so $\alpha$ is a partition of $t$. 

(i) Suppose $c=8$, $d=9$, and $e=3$. Then 
$$
Y_{[8,9]}^{(3)}=E_{1',8}E_{2',8}E_{3',8}E_{1',9}E_{2',9}E_{3',9}.
$$ 

(ii) Suppose $c=10$, $d=11$, and $e=6$. Then 
$$
Y_{[10,11]}^{(6)}=E_{4',10}E_{5',10}E_{6',10}E_{6',11}.
$$ 

(ii) Suppose $c=12$, $d=13$, and $e=7$. Then 
$$
Y_{[10,11]}^{(7)}=E_{7',12}E_{7',13}.
$$

If in each case we insert $i_{e;a,b}$ into the box $(a,b)$ of the Young diagram, the cases (i),(ii),(iii) are illustrated by the following pictures:
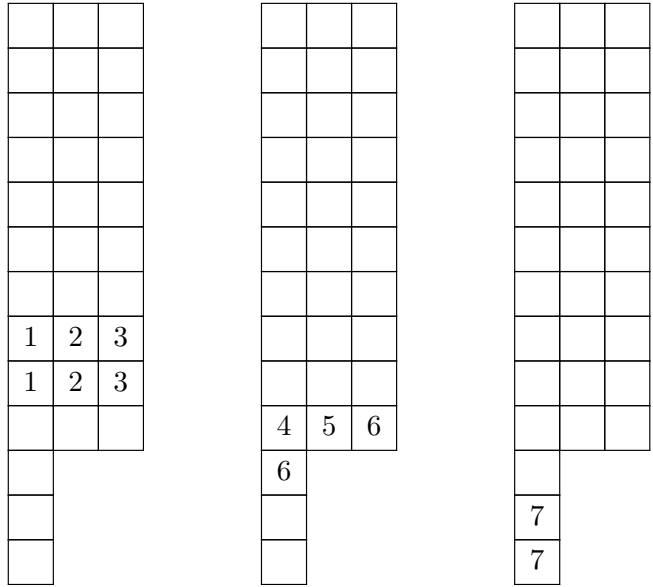
\vspace{3mm}
$$
\begin{ytableau}
\, &  &    \\ 
 &  &   \\
 &  &   \\
 &  &  \\
 &  &  \\ 
 &  & \\
 &  & \\
1  & 2 & 3\\
1   &2  & 3\\
   & &  \\
   \\
  \\
  \\
\end{ytableau}
\qquad\qquad
\begin{ytableau}
\, &  &    \\ 
 &  &   \\
 &  &   \\
 &  &  \\
 &  &  \\ 
 &  & \\
 &  & \\
  &  & \\
   &  & \\
 4   & 5 & 6\\
  6  \\
    \\
    \\
\end{ytableau}
\qquad \qquad
\begin{ytableau}
\, &  &    \\ 
 &  &   \\
 &  &   \\
 &  &  \\
 &  &  \\ 
 &  & \\
 &  & \\
  &  & \\
   &  & \\
    &  & \\
    \\
7    \\
 7   \\
\end{ytableau}
$$
\vspace{5mm}
\end{example}

\begin{lemma}\label{LWeightcd}
Suppose that $Y^{(e)}_{[c,d]}v_+\neq 0$. Then $Y^{(e)}_{[c,d]}v_+$ is a weight vector of weight 
$$
(\alpha-\sum_{b=c}^d\alpha_b\epsilon_b|\sum_{(a,b)\in \alpha_{[c,d]}}\epsilon_{i_{e;a,b}}).
$$
\end{lemma}
\begin{proof}
This is clear from (\ref{EWtAction}). 
\end{proof}

For every $(a,b)\in \alpha_{[c,d]}$ we set 
$$
f_{a,b}:=b+\sharp\{(a_1,b_1)\in\alpha_{[c,d]}\mid a_1>a\ \text{and}\ i_{e;a_1,b_1}=i_{e;a,b}\},
$$
considered as a positive integer or as an element of $\k$ depending on the context.

\begin{lemma} \label{LGoodLower}
We have $$X^{(e)}_{[c,d]}Y^{(e)}_{[c,d]}v_+=\bigg( \prod_{(a,b)\in \alpha_{[c,d]}} f_{a,b}\bigg) v_+.$$ In particular, 
if $\alpha_{c}+d-c<p$ then $Y^{(e)}_{[c,d]}v_+\neq 0$.
\end{lemma}
\begin{proof}
Note that  for every $(a,b)\in \alpha_{[c,d]}$ we have 
$$
f_{a,b}\leq f_{c,\alpha_c}=\alpha_c+d-c<p
$$
by assumption. So, considered as an element of $\k$, every integer $f_{a,b}$ is non-zero. Therefore it suffices to prove the first claim.  Recall the notation (\ref{EXYLoc3}).  
To prove the first claim, it is sufficient to prove that for any $(a,b)\in \alpha_{[c,d]}$ we have 
\begin{equation}\label{EMainComm}
E_{a,i_{e;a,b}'}E_{i_{e;a,b}',a}E_{i_{e;a,b+1}',a}\cdots E_{i_{e;a,\alpha_a}',a}v_{a+1}=f_{a,b}E_{i_{e;a,b+1}',a}\cdots E_{i_{e;a,\alpha_a}',a}v_{a+1}.
\end{equation}

In this paragraph we prove that for any $c\leq a\leq d$ and $1\leq j\leq t$, we have 
\begin{equation}\label{EZero}
E_{a,j'}v_{a+1}=0.
\end{equation}
Indeed, note that $E_{a,j'}v_+=0$ since $v_+$ is a highest weight vector. 
Moreover, note that $Y_{a+1}^{(e)}\cdots Y_d^{(e)}$ is a product of $E_{i_{e;a_1,b}',a_1}$ with $a_1>a$, so 
$E_{a,j'}$ supercommutes with such a factor unless $i_{e;a_1,b}=j$. For such exceptional factors 
$E_{i_{e;a_1,b}',a_1}$, we have 
$$
E_{a,j'}E_{i_{e;a_1,b}',a_1}=-E_{i_{e;a_1,b}',a_1}E_{a,j'}+E_{a,a_1}.
$$
Now $E_{a,a_1}$ commutes with every term appearing in $Y_{a+1}^{(e)}\cdots Y_d^{(e)}$, and $E_{a,a_1}v_+=0$ since $v_+$ is a highest weight vector and $a<a_1$. So we can supercommute $E_{a,j'}$ all the way to $v_+$ to get 
$$
E_{a,j'}v_{a+1}=\pm Y_{a+1}^{(e)}\cdots Y_d^{(e)} E_{a,j'}v_+=0,
$$
completing the proof of (\ref{EZero}). 

In this paragraph we prove that for any $c\leq a\leq d$ and $1\leq b_1<b_2\leq \alpha_a$, we have 
\begin{equation}\label{EClever}
E_{i_{e;a,b_2}',i_{e;a,b_1}'}v_{a+1}=0.
\end{equation}
Indeed, 
$E_{i_{e;a,b_2}',i_{e;a,b_1}'}v_{a+1}=E_{i_{e;a,b_2}',i_{e;a,b_1}'}Y_{a+1}^{(e)}\cdots Y_d^{(e)} v_+$. If we commute $E_{i_{e;a,b_2}',i_{e;a,b_1}'}$ all the way past $Y_{a+1}^{(e)}\cdots Y_d^{(e)}$, we get zero since $E_{i_{e;a,b_2}',i_{e;a,b_1}'}v_+=0$ by Lemma~\ref{L'}. On the other hand, $E_{i_{e;a,b_2}',i_{e;a,b_1}'}$ will produce a non-trivial commutator with the factor $E_{i_{e;a_1,b}',a_1}$ appearing in $Y_{a+1}^{(e)}\cdots Y_d^{(e)}$ only if $i_{e;a,b_1}=i_{e;a_1,b}$, in which case the commutator equals $E_{i_{e;a,b_2}',a_1}$. 
Since $b_1<b_2$, we have that 
$i_{e;a_1,b}=i_{e;a,b_1}<i_{e;a,b_2}\leq e$ so
$E_{i_{e;a,b_2}',a_1}$ also appears in the product $Y^{(e)}_{a_1}$ in the guise $E_{i_{e;a_1,b}',a_1}$. This means that the commutator term will contain $E_{i_{e;a,b_2}',a_1}^2$, hence it is zero by (\ref{EOddSquare}). 

We now prove (\ref{EMainComm}). Note that (\ref{ECommE}) implies that
\begin{align*}
&E_{a,i_{e;a,b}'}E_{i_{e;a,b}',a}E_{i_{e;a,b+1}',a}\cdots E_{i_{e;a,\alpha_a}',a}v_{a+1}
\\=\,\,
&(E_{a,a}+E_{i_{e;a,b}',i_{e;a,b}'})E_{i_{e;a,b+1}',a}\cdots E_{i_{e;a,\alpha_a}',a}v_{a+1}
\\
&+\sum_{k=b+1}^{\alpha_a}\pm E_{i_{e;a,b}',a}E_{i_{e;a,b+1}',a}\cdots E_{i_{e;a,k-1}',a}E_{i_{e;a,k}',i_{e;a,b}'}
E_{i_{e;a,k+1}',a}\cdots E_{i_{e;a,\alpha_a}',a}v_{a+1}
\\
&\pm  E_{i_{e;a,b}',a}E_{i_{e;a,b+1}',a}\cdots E_{i_{e;a,\alpha_a}',a}E_{a,i_{e;a,b}'}v_{a+1}.
\end{align*}
We denote the three summands in the the right hand side by $\mathbf{S}_1$, $\mathbf{S}_2$ and $\mathbf{S}_3$, respectively. Note that $\mathbf{S}_3=0$ by (\ref{EZero}). 

Furthermore, we have that 
$$
\mathbf{S}_1=f_{a,b}E_{i_{e;a,b+1}',a}\cdots E_{i_{e;a,\alpha_a}',a}v_{a+1}.
$$
Indeed, it is easy to see that 
$$
E_{a,a}E_{i_{e;a,b+1}',a}\cdots E_{i_{e;a,\alpha_a}',a}v_{a+1}
=b\,E_{i_{e;a,b+1}',a}\cdots E_{i_{e;a,\alpha_a}',a}v_{a+1}
$$
and 
$$
E_{i_{e;a,b}',i_{e;a,b}'}E_{i_{e;a,b+1}',a}\cdots E_{i_{e;a,\alpha_a}',a}v_{a+1}=
C\,
E_{i_{e;a,b+1}',a}\cdots E_{i_{e;a,\alpha_a}',a}v_{a+1}
$$
where 
$$C=\sharp\{(a_1,b_1)\in\alpha_{[c,d]}\mid a_1>a\ \text{and}\ i_{e;a_1,b_1}=i_{e;a,b}\}$$ using the fact that by definition we have $i_{e;a,b}\not\in\{i_{e;a,b+1},\cdots ,i_{e;a,\alpha_a}\}$. 

It remains to prove that every summand in $\mathbf{S}_2$ is zero. Let $b+1\leq k\leq \alpha_a$. Then 
\begin{align*}
&E_{i_{e;a,b+1}',a}\cdots E_{i_{e;a,k-1}',a}E_{i_{e;a,k}',i_{e;a,b}'}
E_{i_{e;a,k+1}',a}\cdots E_{i_{e;a,\alpha_a}',a}v_{a+1}
\\=\,
&E_{i_{e;a,b+1}',a}\cdots E_{i_{e;a,k-1}',a}
E_{i_{e;a,k+1}',a}\cdots E_{i_{e;a,\alpha_a}',a}E_{i_{e;a,k}',i_{e;a,b}'}v_{a+1}
\end{align*}
which is zero since $E_{i_{e;a,k}',i_{e;a,b}'}v_{a+1}=0$ by (\ref{EClever}). 
\end{proof}


We now define a non-negative integer $k$ and positive integers $b_0\geq b_1\geq\dots\geq b_k$ and $a_0< a_1<\dots< a_{k+1}:=h+1$ recursively as follows. We set $a_0:=m+1$ and $b_0:=\alpha_{a_0}$. For $j=1,2,\dots$, on the $j$th step, 
if $a_{j-1}+p-b_{j-1}> h$ then we define $k:=j-1$, $a_j:=h+1$ and stop; if $a_{j-1}+p-b_{j-1}\leq h$, then 
we define $a_j:=a_{j-1}+p-b_{j-1}$ and 
$b_j:=\alpha_{a_j}$.

Note that $$\alpha_{a_j}+a_{j+1}-1-a_j=b_j+a_{j+1}-1-a_j= p-1$$ 
for $0\leq j<k$, and 
$\alpha_{a_k}+a_{k+1}-1-a_k< p$ by definition. So for all $j=0,1,\dots,k$, we have 
\begin{equation}\label{ECond<p}
\alpha_{a_j}+a_{j+1}-1-a_j<p.
\end{equation}

\begin{lemma}\label{Claim1} We have $b_0+b_1+\dots+b_k\leq n$. 
\end{lemma}
\begin{proof}
We will use the description of $t$ given in (\ref{ERange}). 

By Lemma~\ref{L<p}, we have $b_0\leq n$, so 
$$|\alpha|=t\leq t^{b_0}_p(m,n)=b_0(m+1)+(p-b_0)(n+1-b_0)-1.
$$
On the other hand, note that 
$$\alpha_1+\dots+\alpha_{m+1}\geq (m+1)\alpha_{m+1}=(m+1)b_0,$$ so, using the bound $\alpha_a\ge \alpha_{a_{j}}=b_{j}$ for $a_{j-1}<a\le a_{j}$, we have
$$
|\alpha|-(m+1)b_0\geq \sum_{a=m+1}^h\alpha_a\geq \sum_{j=1}^kb_j(a_j-a_{j-1})  =\sum_{j=1}^kb_j(p-b_{j-1})
$$
It follows from the two inequalities that
$$
\sum_{j=1}^kb_j(p-b_{j-1})\leq (n+1-b_0)p-b_0(n+1-b_0)-1,
$$
which is equivalent to
$$
p(-n-1+\sum_{j=0}^kb_j)< \sum_{j=1}^kb_{j-1}b_j -b_0(n+1-b_0).
$$
Now write $b_0+b_1+\dots+b_k= n+1+u$ for some $u\in\mZ$, so that we need to show that $u<0$. Using $u$ we can rewrite the above inequality as
$$
pu\;<\; \sum_{j=1}^kb_{j-1}b_j -b_0(\sum_{j=1}^kb_j-u),
$$
which is equivalent to
$$(p-b_0)u\;<\;\sum_{j=1}^k (b_{j-1}-b_0)b_j.$$
Now $p-b_0$ is strictly positive, by Lemma~\ref{L<p}, while the right-hand side is negative or zero, showing that indeed $u<0$.
\end{proof}

Recalling (\ref{EXYLoc1}),(\ref{EXYLoc2}) we now consider the following elements of $U_\k$:
\begin{align*}
Y&:=Y^{(b_0)}_{[a_0,a_1-1]}Y^{(b_0+b_1)}_{[a_1,a_2-1]}\cdots Y^{(b_0+b_1+\dots+b_k)}_{[a_k,a_{k+1}-1]},
\\
X&:=X^{(b_0)}_{[a_0,a_1-1]}X^{(b_0+b_1)}_{[a_1,a_2-1]}\cdots X^{(b_0+b_1+\dots+b_k)}_{[a_k,a_{k+1}-1]}.
\end{align*}
Note using Lemma~\ref{Claim1} that each $Y_{[c,d]}^{(e)}$ appearing in the product $Y$ satisfies the assumption (\ref{EAssumption}), so $Y$ actually makes sense, and similarly so does $X$.

Taking into account Lemma~\ref{LWeightcd}, it follows that the vector 
$Yv_+$ lies in the weight space $L_{\lambda|\mu}$, where 
$$\lambda=\sum_{a=1}^m\alpha_a\epsilon_a\in\Lambda(m)$$ 
and 
$$\mu=
\sum_{l=0}^k\sum_{(a,b)\in\alpha_{[a_l,a_{l+1}-1]}}\epsilon_{i_{b_0+\dots+b_l;a,b}}.
$$ 
By Lemma~\ref{Claim1}, each $i_{b_0+\dots+b_l;a,b}$ appearing in this sum for $\mu$ satisfies 
$$1\leq i_{b_0+\dots+b_l;a,b}\le b_0+\dots+b_l \leq n$$ Therefore $(\lambda|\mu)\in \Lambda(m|n,t)$ as required. 

To complete the proof of Proposition~\ref{PMainOld}(i), it remains to prove that $Yv_+\neq 0$. To check this we first note that 
$$
XYv_+=\pm X^{(b_0)}_{[a_0,a_1-1]}Y^{(b_0)}_{[a_0,a_1-1]}X^{(b_0+b_1)}_{[a_1,a_2-1]}Y^{(b_0+b_1)}_{[a_1,a_2-1]}\cdots X^{(b_0+b_1+\dots+b_k)}_{[a_k,a_{k+1}-1]}Y^{(b_0+b_1+\dots+b_k)}_{[a_k,a_{k+1}-1]}v_+. 
$$
Indeed, if $0\leq l_1\neq l_2\leq k$ then for the terms $E_{j_1,i_1'}$ appearing in 
$X^{(b_0+b_1+\dots+b_{l_1})}_{[a_{l_1},a_{l_1+1}-1]}$ and the terms $E_{i_2',j_2}$ appearing in 
$Y^{(b_0+b_1+\dots+b_{l_2})}_{[a_{l_2},a_{l_2+1}-1]}$ we have $j_1\neq j_2$ and $i_1\neq i_2$, since
$$a_{l_1}\le j_1 < a_{l_1+1}\quad\mbox{and}\quad b_0+b_1+\cdots +b_{l_1-1}< i_1\le b_0+b_1+\cdots+ b_{l_1}, $$
with similar bounds for $l_2$. So $E_{j_1,i_1'}$ and $E_{i_2',j_2}$ anticommute. Hence 
$X^{(b_0+b_1+\dots+b_{l_1})}_{[a_{l_1},a_{l_1+1}-1]}$ and $Y^{(b_0+b_1+\dots+b_{l_2})}_{[a_{l_2},a_{l_2+1}-1]}$ commute up to a sign. 

It remains to apply Lemma~\ref{LGoodLower} to $X^{(b_0+b_1+\dots+b_l)}_{[a_l,a_{l+1}-1]}Y^{(b_0+b_1+\dots+b_l)}_{[a_l,a_{l+1}-1]}v_+$ for $l=0,1,\dots,k$, noting that 
$\alpha_{a_l}+a_{l+1}-1-a_l<p$ for all such $l$ by (\ref{ECond<p}). 

\begin{example}
We continue with Example~\ref{Ex2}. In that case we have $k=2$, $a_0=8$, $a_1=10$, $a_2=12$, $a_3=14$ and $b_0=3$, $b_1=3$, $b_2=1$. So
$$
Y=Y_{[8,9]}^{(3)}Y_{[10,11]}^{(6)}Y_{[12,13]}^{(7)}. 
$$
Note that $Y_{[8,9]}^{(3)},Y_{[10,11]}^{(6)},Y_{[12,13]}^{(7)}$ are described explicitly in Example~\ref{Ex2}. Note that $Yv_+$ lies in the weight space $L_{\lambda|\mu}$, where
$\lambda=(3^7)\in\Lambda(7)$ and $\mu=(2,2,2,1,1,6,2)\in\Lambda(7)$. 
The lowering operator $Y$ in this example is illustrated by the following picture:
\vspace{3mm}
$$
\begin{ytableau}
\, &  &    \\ 
 &  &   \\
 &  &   \\
 &  &  \\
 &  &  \\ 
 &  & \\
 &  & \\
1  & 2 & 3\\
1   &2  & 3\\
4   & 5& 6 \\
6   \\
7  \\
7  \\
\end{ytableau}
$$
\vspace{5mm}
\end{example}

\subsection{Upper bound}
\label{SSKevinCounterExamples}
In this section we prove Proposition~\ref{PMainOld}(ii). We use the reformulation of the proposition in terms of the general linear supergroup from~\ref{SchurGL}.

Fix $r\in\mZ_{>0}$. The standard Borel subgroup of $GL(r|r)$ has even positive simple roots
$$\varepsilon_{i}-\varepsilon_{i+1},\quad \mbox{and}\quad \varepsilon'_{i}-\varepsilon'_{i+1},\quad 1\le i<r$$
and odd positive root $\varepsilon_r-\varepsilon'_1$.
For integers $r> a\ge b\ge 1$, we have a Borel subgroup $B_{a,b}$ with same even roots as above, but for which the odd simple positive roots are
$$\varepsilon_{a}-\varepsilon'_1,\varepsilon'_{b}-\varepsilon_{a+1},\varepsilon_r-\varepsilon'_{b+1}.$$

\begin{lemma}\label{LemBK}
For a given $\alpha\in \Lambda^+(r)$, the highest weight of $L(\alpha|0)$ with respect to $B_{a,b}$ is of the form
$$\sum_{i=1}^a\alpha_i\varepsilon_i +\sum_{j=1}^b\beta_j\varepsilon'_j+\sum_{l=1}^{r-a}\gamma_l\varepsilon_{a+r}$$
for some $\beta_j\in\mN$ and $\gamma_l\in\mZ$. If there exists $1\le l\le t-a$ for which $\gamma_l>0$ and $\gamma_{l'}=0$ for all $l'>l$, then
$$L(\alpha|0)_{\lambda|\mu}=0$$
for all $(\lambda,\mu)\in \Lambda(a|b)\subset \Lambda(r|r)$.
\end{lemma}
\begin{proof}
The expression for the highest weight with respect to $B_{a,b}$ follows immediately from the procedure of odd reflections, see \cite[Lemma~4.2]{BK}.

Clearly, in the partial order corresponding to $B_{a,b}$, a weight $(\lambda,\mu)\in\Lambda(a|b)$ is not lower than the described highest weight of $L(\alpha|0)$ under the imposed condition on $\{\gamma_l\}$.
\end{proof}

Recall that we have fixed $m\ge n>0$.
\begin{example}\label{Exalpha}
Consider $1\le \indexi\le \min(p-1, n+1)$ and write 
$$n-\indexi+1=a\indexi+b,\quad \mbox{for } a\in\mN\mbox{ and }0\le b<\indexi.$$
For $r\ge m+1+(a+1)(p-\indexi)$, we can consider the partition
$$\alpha_{(n,s)}:=(\indexi^{m+1+a(p-\indexi)},b^{p-\indexi})\in \Lambda^+(r, t_p^s(m,n)+1).$$
Of course $\alpha_{(n,s)}$ actually depends on $m,n,\indexi,p$, but we consider $m,p$ as constant, motivating the notation. We have for example
$$\alpha_{(n,1)}=(1^{m+1+(p-1)n}),\;\mbox{ and }\;\quad\alpha_{(n,n+1)}=((n+1)^{m+1}).$$
Then $L(\alpha_{(n,s)}|0)_{\lambda|\mu}=0$
for all $(\lambda,\mu)\in \Lambda(m|n)$, by Lemma~\ref{LemBK}.

Indeed, the calculation of the highest weight with respect to $B_{m,n}$, following the rule in \cite[Lemma~4.2]{BK}, can for instance be structured as follows. First we can compute the highest weight with respect to $B_{m,1}$, which is
$$\begin{cases}
(\alpha_{(n-1,\indexi)}\,|\, p-\indexi)&\mbox{ if }\indexi\le n\\
((n+1)^m,n|1) &\mbox{ if }\indexi= n+1.
\end{cases}$$
Passing to $B_{m,2}$ (in case $n>1$), $B_{m,3}$ etc. involves the exact same computation as above. In particular, the highest weight with respect to $B_{m,n-\indexi+1}$ (if $\indexi\le n$) is
$$(\indexi^{m+1}|(p-\indexi)^{n-\indexi+1})$$
and subsequently, the highest weight with respect to $B_{m,n}$ is
$$(\indexi^{m},1|(p-\indexi)^{n-\indexi+1}, 1^{\indexi-1}).$$
In other words (with notation from Lemma~\ref{LemBK}), we have $\gamma_1=1$ and $\gamma_l=0$ for $l>1$.

\end{example}

The following corollary implies Proposition~\ref{PMainOld}(ii). 
\begin{corollary}
For $1\le \indexi \le \min(p-1 ,n+1)$ and $r=t^\indexi_p(m,n)+1$ as in \eqref{Es}, 
there exists a $p$-restricted partition $\alpha\in\Lambda^+(r)$ such that $L(\alpha|0)_{\lambda|\mu}= 0$ for all $(\lambda,\mu)\in\Lambda(m|n,r)$.
\end{corollary}
\begin{proof}
For $\alpha=\alpha_{(n,s)}$ as in Example~\ref{Exalpha}, it only remains to observe that  $\alpha$ is $p$-restricted, which follows from $\indexi<p$.
\end{proof}

\section{Surjectivity}\label{SecSurj}

In this section we study surjectivity of the morphism $\Phi^{(r)}_{m,n}$ from equation~\eqref{EqPhi}.

\subsection{Main results}

Recall the integer $r_{p}(m,n)$ which was defined and determined in Section~\ref{SecInj}.

\begin{theorem}\label{ThmSurj}
Assume that $p>2$. 

\begin{enumerate}
\item[{\rm (i)}] Whenever $$r_p(m,n)\;< \; r\;\le \; m+n+mn,$$
the morphism $\Phi_{m|n}^{(r)}$
is neither injective nor surjective.
\item[{\rm (ii)}]  For the morphism $\Phi_{m|n}^{(r)}$
to be surjective for all $r\in\mN$, we need
$$mn=0,\quad\mbox{or}\quad m+n<p.$$
\item[{\rm (iii)}]  The morphism $\Phi^{(r)}_{1|1}$ is surjective for all $r$.
\end{enumerate} 
\end{theorem}

It is an open question whether the necessary condition for surjectivity in Theorem~\ref{ThmSurj}(ii) is also sufficient. However, for $p=3$, the combination with Theorem~\ref{ThmSurj}(iii) shows it is:

\begin{corollary}\label{CNew}
For $p=3$, the morphism $\Phi_{m|n}^{(r)}$
is surjective for all $r\in\mN$ if and only if $mn=0$ or $m=1=n$.
\end{corollary}

The remainder of the section is devoted to the proof of Theorem~\ref{ThmSurj} and some similar considerations for the orthosymplectic supergroup.

\subsection{Failure of surjectivity}
In this section we prove parts (i) and (ii) of Theorem~\ref{ThmSurj}.

\begin{lemma}\label{StandLem}
Consider a homomorphism $\phi:A\to B$ between free abelian groups of finite rank. 
\begin{enumerate}
\item[{\rm (i)}] We have an inclusion of vector spaces
$$\k\otimes \ker\phi\;\subset\; \ker(\operatorname{id}_\k\otimes \phi),$$
which is an equality if $p=0$.
\item[{\rm (ii)}] For a field extension $\mathbb{K}/\k$, we have the equality
$$\mathbb{K}\otimes_\k (\ker(\operatorname{id}_\k\otimes \phi))\;=\; \ker(\operatorname{id}_\mathbb{K}\otimes \phi).$$
\end{enumerate}
\end{lemma}
\begin{proof}
Observe that $\im\phi$ is free, as a subgroup of $B$.
It follows that we have a short exact sequence
$$0\to \k\otimes \ker\phi\to \k\otimes A\to \k\otimes \im\phi\to 0.$$
Since $\operatorname{id}_\k\otimes \phi$ factors through $\k\otimes \im\phi$, the inclusion in (i) follows. 
The equality in (i) in characteristic zero follows since $\k$ is then flat over $\mZ$.

Part (ii) is obvious, since $\mathbb{K}\otimes_\k-$ is exact.
\end{proof}

Denote by $\delta^r_{m|n}(\k)$ the dimension of the right-hand side in~\eqref{EqPhi}:
$$\delta^r_{m|n}(\k):=\dim_\k\End_{GL({m|n})}(V_{m|n}^{\otimes r}).$$

\begin{prop}\label{PropBound} Fix $r,m,n$.
The value $\delta^{r}_{m|n}(\k)$ depends only on $p=\mathrm{char}(\k)$, and
$$\delta^r_{m|n}(\k)\;\ge\; \delta^r_{m|n}(\mathbb{Q}).$$
\end{prop}
\begin{proof}
Set 
$$A=\End(V_{m|n,\mZ}^{\otimes r})\simeq \mZ^{(m+n)^{2r}}.$$
Then there exists a homomorphism
$$f:A\to A^l$$
of abelian groups, for some $l\in\mN$, such that 
$$\End_{GL({m|n})}( V_{m|n}^{\otimes r})\;\simeq\;\ker(\k\otimes f).$$
The statements now follow from Lemma~\ref{StandLem}.
\end{proof}




\begin{proof}[Proof of Theorem~\ref{ThmSurj}(i) and (ii)] By definition, $\Phi^{(r)}_{m|n}$ is not injective for
\begin{equation}\label{EqS1}
r_p(m,n)< r
\end{equation} while, by Remark~\ref{RP}(i), $\Phi^{(r)}_{m|n}$ is injective (and surjective) for $\k=\mQ$ under the assumption 
\begin{equation}\label{EqS2}r\le m+n+mn.\end{equation} It thus follows from Proposition~\ref{PropBound} that under \eqref{EqS2} the dimension of the right-hand side in~\eqref{EqPhi} is at least that of the left-hand side. If additionally \eqref{EqS1} holds, it thus follows that $\Phi$ is not surjective as it is not injective. This proves part (i).

For part (ii), we can focus on the case $m\ge n>0$. By part (i) and Lemma~\ref{LUpper} a necessary condition for surjectivity is
$$r_p(m,n)=m+n+mn=t^{n+1}_p(m,n).$$
By Theorem~\ref{TMain0}, for instance using Remark~\ref{RemMain}, this condition is equivalent to the condition $p< m+n$.
\end{proof}

\subsection{Invariant theory for $GL(1|1)$} Here we always assume $p\not=2$.
Theorem~\ref{ThmSurj}(iii) follows from the following theorem.
\begin{theorem}\label{Thm11}
For $\Phi^{(r)}_{1|1}:\k S_r\to \End_{GL(1|1)}(V_{1|1}^{\otimes r})$ in \eqref{EqPhi}, we have
$$\dim_{\k}\End_{GL(1|1)}(V_{1|1}^{\otimes r})\;=\; \binom{2r-2}{r-1}\;=\; \dim_{\k}\im\Phi^{(r)}_{1|1},$$
for any $p\not=2$.

\end{theorem}
\begin{proof}
Here we demonstrate the left equality, while the right equality follows from Lemmata~\ref{Lem1}(iii) and~\ref{Lem2}(iii) below.

To prove the claim, we can use some basic facts about the representation theory of $GL(1|1)$. For $a,b\in\mZ$, the simple module $L(a|b)$ is projective if and only if it is two-dimensional (with lowest weight $a-1|b+1$) if and only if $a+b$ is not zero in $\k$. In the other cases the simple module is one-dimensional and its projective cover $P(a|b)$ is self-dual of Loewy-length $3$ with middle layer $L(a+1|b-1)\oplus L(a-1|b+1)$.

Consequently, if $p$ does not divide $r$ (or $p=0$), from elementary character comparison (where $D$ denotes the dimension of the endomorphism algebra)
$$V^{\otimes r}\;\simeq\; \bigoplus_{i=0}^{r-1} L(r-i|i)^{\oplus \binom{r-1}{i}},\quad\mbox{so}\;\, D=\sum_{i=0}^{r-1} \binom{r-1}{i}^2.$$
Since $V=L(1|0)$ is projective, it follows that all powers $V^{\otimes r}$ are projective.
If $p$ divides $r$, we thus similarly find
$$V^{\otimes r}\;\simeq\; \bigoplus_{i=1}^{r-1} P(r-i|i)^{\oplus \binom{r-2}{i-1}},\quad\mbox{so}\;\, D=\sum_{i=1}^{r-1} \binom{r-2}{i-1}\binom{r}{i}.$$
Both sums are equal to the claimed dimension, by Vandermonde's identity.
\end{proof}

\begin{remark}
Since the second equality in the theorem will also be obtained by explicit calculation, we also obtain rather explicit descriptions of the kernel of \eqref{EqPhi}. For instance, if $p$ does not divide $r$, the kernel, as a submodule of the regular module, corresponds to the radicals of the projective covers of the simple modules labelled by ($p$-regularisations of) hook partitions and the entire projective covers for the other simple modules.
\end{remark}

\subsubsection{} For $0<j\le r$, we denote by $\hook(r;j-1)$ the hook partition of size $r$ and length $j$. In other words 
$$\hook(r;i)\,=\,(r-i, 1^{i}),\qquad\mbox{for $0\le i<r$}.$$

We denote the $p$-regularisation (as defined in \cite[Section~6.3]{JK}) of a partition $\lambda$ by $\tR(\lambda)$. If $p=0$ then by convention $\tR(\lambda)=\lambda$.

\begin{lemma}
For $r\in\mZ_{>0}$ and $0\le j<i<r$, we have $\tR(\hook(r;i))=\tR(\hook(r;j))$ if and only if $p$ divides $r$ and 
$$i=(p-1)\frac{r}{p},\quad\mbox{and}\quad j=i-1.$$
\end{lemma}
\begin{proof}

For $0\le i<r$, we can write
$$i=(p-1)a+b\quad\mbox{and}\quad i+1 = (p-1)a_1+b_1$$
for unique $a,b,a_1,b_1\in\mN$ with $0<b,b_1<p$. For instance $a_1=a$ and $b_1=b+1$ whenever $b<p-1$. It then follows easily that
$$\tR(\hook(r;i))\;=\;\begin{cases}
(r-i, (a+1)^b, a^{p-1-b}) & \mbox{ if $a+1<r-i$}\\
((a_1+1)^{b_1}, a_1^{p-1-b_1}, r-i-1)& \mbox{ if $a+1\ge r-i$}.
\end{cases}$$

The only case in which the regularisations could become identical is if for $\hook(r;j)$ we are in the first line and for $\hook(r;i)$ on the second line (up to $i\leftrightarrow j$). We use the notation $j=(p-1)a'+b'$. Clearly we must have $b_1=1$ and either $a_1+1=r-i$ (in which case $b'=p-1$) or $a_1=r-i$. The first option for $a_1$ is not relevant, because it would require $r-j=a_1+1=r- i$ and hence $j=i$, a contradiction. So we have $a_1=r-i$ (implying already $j=i-1$), which, together with $b_1=1$, shows indeed $pi=(p-1)r$.
\end{proof}

\subsubsection{}

Since $V=\k^{1|1}$, we have
$$V^{\otimes r}\;\simeq\; \bigoplus_{a+b=r} M^{a|b}$$
as an $\k S_r$-module.

By the Littlewood-Richardson rule, for $0<a,b$, the module $M^{a|b}$ has a filtration with subquotients the Specht modules $S^{\hook(r;b)}$ and $S^{\hook(r;b-1)}$.

\begin{lemma}\label{Lem1}
Fix $r\in\mN$ not divisible by $p$, or assume $p=0$.
\begin{enumerate}
\item[{\rm (i)}]  For $a+b=r$ with $0<a,b$, in $\Rep_\k S_r$
$$M^{a|b}\;\simeq\; S^{\hook(r;b)}\oplus S^{\hook(r;b-1)}.$$
\item[{\rm (ii)}]  For $0\le i<r$, the Specht module $S^{\hook(r;i)}$ is simple with 
$$\dim_\k S^{\hook(r;i)}\;=\;\binom{r-1}{i}.$$
\item[{\rm (iii)}]  The dimension of the image of $\Phi_{1|1}^{(r)}$
is $\binom{2r-2}{r-1}$.
\end{enumerate}
\end{lemma}
\begin{proof}
Part (i) follows from the fact that both Specht modules are in different blocks and self-duality. All dimension calculations are then immediate.
\end{proof}

\begin{lemma}\label{Lem2}
Fix $r\in\mN$ divisible by $p>2$.
\begin{enumerate}
\item[{\rm (i)}]  There are non-isomorphic simple $\k S_r$-modules $\{D_i\,|\, 0\le i\le r-2\}$ with
$$\dim_\k D_i \;=\;\binom{r-2}{i}$$
such that $S^{\hook(r;0)}\simeq D_0$ is the trivial module, $S^{\hook(r;r-1)}\simeq D_{r-2}$ is the sign module, and there are short exact sequences
$$0\to D_{i-1} \to S^{\hook(r,i)}\to D_i\to 0,\quad\mbox{for}\quad 0<i<r-1.$$
\item[{\rm (ii)}]  We have $M^{r|0}\simeq D_0$, $M^{0|r}\simeq D_{r-2}$ and the socle filtrations
$$M^{r-1|1}:\begin{array}{c}D_0\\ D_1\\ D_0\end{array},\quad M^{r-i|i}: \begin{array}{c}D_{i-1}\\ D_{i-2}\oplus D_{i}\\ D_{i-1}\end{array}\;\mbox{ for }1<i<r-1,\quad\mbox{and}\quad M^{1|r-1}:\begin{array}{c}D_{r-2}\\ D_{r-3}\\ D_{r-2}\end{array}.
$$
\item[{\rm (iii)}]  The dimension of the image of $\Phi_{1|1}^{(r)}$
is $\binom{2r-2}{r-1}$.
\end{enumerate}
\end{lemma}
\begin{proof}
By restricting to $S_{r-1}<S_r$ and using Lemma~\ref{Lem1}, we can observe that each $S^{\hook(r;i)}$ can be of length at most two, and in case $i<r-1$ must contain a simple constituent which does not appear in any $S^{\hook(r;j)}$ for $j<i$. Moreover, it follows similarly that $S^{\hook(r;i)}$ and $S^{\hook(r;j)}$ have no simple constituent in common if $|i-j|>1$.

The forms of $M^{r-1|1}$ and $S^{\hook(r-1,1)}$ are well-known and we prove the remaining statements by induction:

Assume that for some $1<i<r-2$, we already know $S^{\hook(r;i-1)}$ has the stated filtration and the simple constituents have the stated dimension. The fact that $M^{r-i|i}$ must be self-dual forces $S^{\hook(r;i)}$ to be one of four possibilities. Either it is $D_{i-2}$, $D_{i-1}$, or is of length $2$ with $D_{i-1}$ in the socle or $D_{i-2}$ in the top. The first two options are excluded by the fact (see first paragraph) that $S^{\hook(r;i)}$ must contain a `new' simple constituent. We also know (again by the first paragraph) that $D_{i-2}$ cannot appear as a constituent in $S^{\hook(r;i)}$. Hence (by defining $D_i$ as the top of $S^{\hook(r-i;i)}$) we find the appropriate filtrations for $S^{\hook(r-i;i)}$, $M^{r-i|i}$ and the stated dimension 
$$\binom{r-2}{i}\;=\; \binom{r-1}{i}-\binom{r-2}{i-1}$$
for $D_i$. Concluding the cases $i=r-2,r-1$ is then easy.

By Lemma~\ref{LemDim}(i), the dimension of the image is at least
$$\sum_{i=0}^{r-2}\left(\dim_\k D_i\right)\dim_\k M^{r-i+1|i+1}=\sum_{i=0}^{r-2}\binom{r-2}{i}\binom{r}{i+1},$$
where equality would be guaranteed if we have isomorphisms between the two-dimensional quotients and submodules of the signed permutation modules. However, by the (already proved) left equality in Theorem~\ref{Thm11}, the above value is also the dimension of the target in \eqref{EqPhi}, so it is the maximal value of the image. Part (iii) follows.
\end{proof}

\begin{remark}
We can easily check that
$$D_i\;\simeq\;\begin{cases}
D^{\tR(\hook(r;i))}&\mbox{for }i<(p-1)r/p\\
D^{\tR(\hook(r;i+1))}&\mbox{for }i\ge (p-1)r/p.
\end{cases}$$
\end{remark}

%
%

\subsection{The orthosymplectic case}

Let $\cB_r(\delta)$ be the Brauer algebra on $r$ strands, with loops evaluated at $\delta\in \k$. We fix $m,n\in \mN$. For $V=\k^{m|2n}$, we have an algebra morphism
\begin{equation}
\label{EqPsi0}
\cB_r(m-2n)\;\to\; \End_{\k}( V^{\otimes r}),
\end{equation}
so that the restriction to $\k S_r\subset\cB_r(m-2n)$ corresponds to \eqref{EqPhi0}, see \cite{LZ}.

\begin{prop}\label{PropOSp1}${}$
\begin{enumerate}
\item[{\rm (i)}]  (\cite{Selecta, Yang}) If $p=0$, then \eqref{EqPsi0} is injective if and only if
$$r\;\le\; m+n+mn.$$
\item[{\rm (ii)}]  If $p>2$, then \eqref{EqPsi0} is not injective if
$$r\;>\;\min(m+n+mn,r_p(m,2n)).$$
\item[{\rm (iii)}]  If $p=2$, and $m$ even, then \eqref{EqPsi0} is not injective if
$$r\;>\; m/2+n.$$
\end{enumerate}
\end{prop}
\begin{proof}
One condition in part (ii) comes from part (i) and Lemma~\ref{StandLem}(1), the other from restriction to $\k S_r\subset\cB_r(m-2n)$. Part (iii) follows again from part (i) and Lemma~\ref{StandLem}(1), by observing that for $p=2$ we can replace $\k^{m|2n}$ by $\k^{0|2(m/2+n)}$.
\end{proof}

\subsubsection{}\label{DLZ} If $p\not=2$, then \eqref{EqPsi0} lifts to
\begin{equation}
\label{EqPsi}
\Psi^{(r)}_{m|2n}\;:\;\cB_r(m-2n)\;\to\; \End_{OSp(V)}( V^{\otimes r}).
\end{equation}
For $\k=\mC$, the morphism $\Psi^{(r)}_{m|2n}$ is always surjective, see~\cite{DLZ, LZ, Se}. 

Using the same dimensional arguments that led to Theorem~\ref{ThmSurj}, one can use the above results to prove that $\Psi^{(r)}_{m|2n}$ is not necessarily surjective in positive characteristic. We just give an example.

\begin{example}
If $p>2$, $m\ge 2n>0$ and $m>2p-3$, then there are $r$ for which $\Psi^{(r)}_{m|2n}$ is not surjective.
\end{example}



\section{Example: $GL(2|1)$}\label{SecEx}
In this section we let $\k$ be of characteristic $p>2$ and we investigate in detail the failure of surjectivity from Theorem~\ref{ThmSurj} in the smallest relevant case.

\subsection{Results}

The following theorem shows in particular that the lack of surjectivity observed in Theorem~\ref{ThmSurj}(i) actually persists to arbitrarily high degrees $r$.

\begin{theorem}\label{Thm21}
Set $V=V_{2|1}$ and $p=3$.
\begin{enumerate}
\item[{\rm (i)}]  If $r= 2+3i$ for some $i\in\mZ_{>0}$, then $\Phi^{(r)}_{2|1}$
is {\bf not} surjective.
\item[{\rm (ii)}]  We have $\dim_\k\End_{GL(2|1)}(V^{\otimes 5})=120$.
\item[{\rm (iii)}]  The image of $\Phi^{(5)}_{2|1}:\k S_5\to \End_{GL(2|1)}(V^{\otimes 5})$  has dimension $119$.
\end{enumerate}

\end{theorem}

The rest of the section is devoted to the proof. For the remainder of the section, we fix $m=2,n=1$, so $V=V_{2|1}.$

\subsection{Direct argument for failure of surjectivity}

Unless further specified, we assume $p>2$.
We denote by $e_s\in \k S_r$ the (unique) idempotent for which $(\k S_r)e_s$ is the projective cover of the sign module.

\begin{lemma}\label{Lem21a}
Assume neither $r$ nor $r-1$ are divisible by $p$.
\begin{enumerate}
\item[{\rm (i)}]  Over $GL(V)$, the inclusion 
$$\wedge^rV\;\hookrightarrow\; V^{\otimes r}$$
of the $S_r$-invariants for the sign-twisted action, is split.
\item[{\rm (ii)}]  If $\Phi^{(r)}_{2|1}$ is surjective, then $\dim_{\k}\Phi(\k S_re_s)=1$.
\end{enumerate} 
\end{lemma}
\begin{proof}
Claim (i) is only non-trivial for $r\ge p$, which is what we focus on.
If neither $r$ nor $r-1$ is divisible by $p$ it is easy to check that the (4-dimensional) module $L:=\Lambda^r V$ is simple. More precisely it is the simple module $L(1,1|r-2)$, with lowest weight $0,0|r$. 

However the weight space $0,0|r$ in $ V^{\otimes r}$ is only 1-dimensional, meaning that $[V^{\otimes r}:L]=1$. By self-duality (or because the quotient $\Lambda^rV$ of $V^{\otimes r}$ is isomorphic to $ L$ too), $L$ must form a direct summmand.

For part (ii), denote the sign module of $\k S_r$ by $D$ and its projective cover by $P=\k S_r e_s$. We observe
$$\wedge^r V\,\simeq\,\Hom_{kS_r}(D, V^{\otimes r})\,\subset\, \Hom_{kS_r}(P, V^{\otimes r})\,\simeq\, \Phi(e_s)(V^{\otimes r}).$$
Hence, under the assumption that $\Phi$ is surjective, the idempotent in $\End_{GL(V)}(V^{\otimes r})$ corresponding to the summand $\wedge^r V$ must be $\Phi(e_s)$ and the above inclusion an equality. Hence every $\k S_r$-module morphism $P\to V^{\otimes r}$ acts trivially on the radical of $P$, from which (ii) follows.
\end{proof}

\begin{lemma}\label{Lem21b}

If $p=3$ and $r\ge 3$, then
 $\dim_{\k}\Phi(\k S_re_s)>1$.
\end{lemma}
\begin{proof}
The $p$-regularisation $\lambda$ of the partition $(1^r)$ of $r$ is $(s,s)$ if $r=2s$, and $(s+1,s)$ if $r=2s+1$. 

Since $\lambda$ is of length $2=m$, the Specht module $S^{\lambda}$ appears as a subquotient of $V^{\otimes r}$, for instance in $M^{(r/2,r/2)|0}$ if $r$ is even. Since the sign module $D^\lambda$ is the top of $S^\lambda$, it follows that the image in the statement of the lemma has at least dimension $\dim_\k S^\lambda>1$.
\end{proof}

\begin{proof}[Proof of Theorem~\ref{Thm21}(ii)]
This is a consequence of Lemmata~\ref{Lem21a}(ii) and~\ref{Lem21b}.
\end{proof}

\subsection{Proof of Theorem~\ref{Thm21}(ii)}
We set $p=3$, $m=2$, $n=1$ and $V=V_{2|1}$.
By Proposition~\ref{PropBound}, to calculate the dimension of the endomorphism algebra we can assume that $\k$ is algebraically closed.

\subsubsection{} 

We have a decomposition
$$V^{\otimes 5} \;\simeq\; S^2(\Lambda^2 V)\otimes V\,\oplus\, \Lambda^2(\Lambda^2 V)\otimes V\,\oplus\, \Lambda^2(S^2 V)\otimes V\,\oplus\, S^2(S^2 V)\otimes V \,\oplus\, ((\Lambda^2V)\otimes (S^2 V)\otimes V)^2.$$
We label the summands by $Q(i)$ for $1\le i\le 5$, where $Q(5)$ thus appears twice in $V^{\otimes 5}$.

Some general useful (easy to verify) facts are:
\begin{enumerate}[label=(\alph*)]
\item The $\mathfrak{gl}(2|1)$-module $\Lambda^2 V$ is free as a $\mathfrak{g}_{-1}$-module, where $\mathfrak{g}_{-1}$ is the subalgebra spanned by the $E_{i'j}$. Consequently, also $Q(1), Q(2)$ and $Q(5)$ are $\mathfrak{g}_{-1}$-free.
\item The module $Q(i)$ is self-dual, for every $1\le i\le 5$.
\item The Casimir operator in $U(\mathfrak{g})$ is central as an element of the hyperalgebra $U_{\k}$. Moreover on a simple $GL(2|1)$-representation $L(a,b|c)$ it acts via the scalar
$$a^2+(b-1)^2-(c+1)^2\in\k.$$
\item If $a-b\le p-2$ (so $a-b\le 1$ for us) and if $(a+c+1)(b+c)$ is not zero in $\k$, then the character of $L(a,b|c)$ equals that of the (simple) Kac module with highest weight $(a,b|c)$ over $\mathbb{C}$. The relevant examples of this are:
$L(3,2|0)$ and $L(1,1|3)$.
\item As $p=3$, the $GL(2|1)$-representation $L(2,2|1)$ is one-dimensional.
\end{enumerate}

We will also freely use the multiplicities of the simple $GL(2|1)$-representations in the (semisimple) analogues of the $Q(i)$ over $\mC$.

\subsubsection{} A quick computation shows that the subquotient $(\wedge^4 V\otimes V)/\wedge^5V$ of $Q(1)$ is an extension of $L(2,1|2)$ and $L(2,0|3)$, revealing in particular the character of both simples. Using this, (e) and more character comparison shows that the simple multiplicities in $Q(1)$ correspond to the following
$$Q(1)\;\simeq\; L(3,2|0)\;\oplus\; L(1,1|3)\;\oplus \;\begin{array}{c}L(2,1|2)\\ L(2,2|1)\oplus L(2,0|3)\\ L(2,1|2)\end{array}.$$
That the above isomorphism then holds follows from (a), (b), (c) and (d). The procedure determines also the character of all simples involved.
\subsubsection{} Using similar arguments we can subsequently prove
\begin{eqnarray*}
Q(2)&=& L(3,1|1)\;\oplus\; \begin{array}{c}L(2,1|2)\\ L(2,2|1)\oplus L(2,0|3)\\ L(2,1|2)\end{array},\\
Q(3)&=& L(3,1|1)\;\oplus\; \begin{array}{c}L(3,2|0)\\ L(4,1|0)\\ L(3,2|0)\end{array},\\
Q(5)&=&  L(3,1|1)^2\;\oplus\; \begin{array}{c}L(3,2|0)\\ L(4,1|0)\\ L(3,2|0)\end{array}\;\oplus\; \begin{array}{c}L(2,1|2)\\ L(2,2|1)\oplus L(2,0|3)\\ L(2,1|2)\end{array}.
\end{eqnarray*}

\subsubsection{} Since to composition $\Gamma^5V\hookrightarrow V^{\otimes 5}\tto S^5V$ from the $S_5$-invariants to the coinvariants is not zero, it follows quickly that $S^5 V$ is a (non-split) extension of $L(5,0|0)$ and $L(2,2|1)$. A tedious calculation shows that $L(2,1|2)$ is not a submodule of $Q(4)$. It then follows from (b) and (c) that
$$Q(4)\;\simeq\;\begin{array}{c}L(2,2|1)\\ L(5,0|0)\oplus L(2,1|2)\\ L(2,2|1)\end{array} \;\oplus\; \begin{array}{c}L(3,2|0)\\ L(4,1|0)\\ L(3,2|0)\end{array}.$$

\subsubsection{}Regardless of the precise structure of the summands of the  $Q(i)$ with the given socle filtrations, the maximal dimension of the endomorphism algebra can be calculated to be 120. This must be an equality by Proposition~\ref{PropBound}.

\subsection{Proof of Theorem~\ref{Thm21}(iii)}

We set $p=3$, $m=2$, $n=1$ and $V=V_{2|1}$.

\subsubsection{}\label{p3S}
The simple $\k S_5$-modules are
$$D^{(5)}, D^{(2,2,1)}\quad |\quad D^{(3,2)}, D^{(4,1)}\quad |\quad D^{(3,1,1)}$$
where we have split the list according to the blocks. The first two blocks are equivalent, via tensoring with the sign module $D^{(3,2)}$. The block decomposition in turn implies that
$$S^{(4,1)}\simeq D^{(4,1)}\quad\mbox{and}\quad S^{(3,1,1)}\simeq D^{(3,1,1)}.$$

\subsubsection{}
By \eqref{ETensWts} and \eqref{EVMIso}, the $\k S_5$-module $V^{\otimes 5}$ contains (projective) summands $M^{(2,2)|1}\simeq M^{(2,2,1)}$ and $M^{(2,1)|2}$.
We claim that 
$$
M^{(2,2,1)}\quad\mbox{contains}\quad P^{(5)}\oplus P^{(4,1)}\oplus P^{(3,1,1)}.
$$
Indeed, the the first summand is obvious. The second follows from $S^{(4,1)}\simeq D^{(4,1)}$ and the Littlewood-Richardson rule. The third summand follows similarly.

The same argument leads to the first two summands in the claim that
$$
M^{(2,1)|2}\quad\mbox{contains}\quad P^{(4,1)}\oplus P^{(3,1,1)}\oplus P^{(2,2,1)},
$$
while the appearance of $P^{(2,2,1)}$ follows from observing that there is a non-zero morphism $M^{(2,1)|2}\to S^{(2,2,1)}$ but no such morphism $M^{(2,1)|2}\to S^{(5)}$.

From Lemma~\ref{LUpper} , we know that the $\k S_5$-module $V^{\otimes 5}$ is not faithful. By Lemma~\ref{LemDim}(ii), the one remaining indecomposable projective $P^{(3,2)}$ therefore cannot appear as a summand.

\subsubsection{} Using the dimensions of the simple modules and the block equivalence from~\ref{p3S}, expressing the dimension of $\k S_5$ yields
$$120\;=\; 2\dim P^{(3,2)}+8\dim P^{(2,2,1)}+36.$$
Since $\k S_5$ is not semisimple it also follows easily that $\dim P^{(3,2)}\ge 6$ and  $\dim P^{(2,2,1)}\ge 9$. Hence these bounds must be equalities and the projectives have (socle) filtrations
$$P^{(4,1)}\;:\; D^{(4,1)}| D^{(3,2)}|D^{(4,1)},\qquad P^{(3,2)}\;:\; D^{(3,2)}| D^{(4,1)}| D^{(3,2)}.$$

Since $P^{(4,1)}$ appears in $ V^{\otimes 5}$, it follows that the submodule $K$ of $P^{(3,2)}$ from Lemma~\ref{LemDim}(i) is the socle $D^{(3,2)}$, which is one-dimensional. The dimension now follows from Lemma~\ref{LemDim}(i).

\newpage

\section{Some further considerations}\label{SecFurther}
In this section we let $p=\mathrm{char}(\k)$ be prime.

\subsection{Towards an adapted first fundamental theorem}

\subsubsection{Construction}\label{Constr}
Consider a homomorphism of free abelian groups of finite rank
$$f:A\to B.$$
 We have an induced morphism of $\mZ[1/p]$-modules
$$\widetilde{f}:\;\mZ[1/p]\otimes A\;\to\;\mZ[1/p]\otimes B$$
and define the subgroup
$$I_p(f)\;:=\; \im (\widetilde{f})\cap B\,\subset\, B.$$

\begin{example}
For $n\in\mZ$ and $f:\mZ\to \mZ$ given by multiplication by $n$ we have $$I_p(f)=\frac{n}{p^i}\mZ\subset \mZ,$$ where $i\in\mN$ is maximal with $p^i|n$.
\end{example}

\begin{lemma}\label{LemApf}
We have:
\begin{enumerate}
\item[{\rm (i)}] $\dim_{\k}\k\otimes I_p(f)\;=\;\mathrm{rank} A -\mathrm{rank}\ker f\;=\; \dim_{\mC}\mC\otimes A-\dim_{\mC}\ker(\mC\otimes f).$
\item[{\rm (ii)}]  The induced morphism
$$\k\otimes I_p(f)\;\to \k\otimes B$$
is injective.
\item[{\rm (iii)}]  If $f$ is a ring homomorphism, then $I_p(f)\subset B$ is a subring.
\end{enumerate}
\end{lemma}
\begin{proof}
The second equality in part (i) is just Lemma~\ref{StandLem}(i). The first equality follows from the observation that the quotient of the free abelian group $I_p(f)$ by its subgroup $\im f$ is $p$-torsion.

For part (ii) we can observe that by construction $I_p(f)/p\to B/p$ is injective.

Part (iii) is immediate.
\end{proof}

\begin{definition}\label{DSigma}
Define $\Sigma_{m|n}^{(r)}$ as the $\k$-algebra $\k\otimes I_p(f)$ for 
$f:\mZ S_r\;\to\; \End(V^{\otimes r}_{m|n,\mZ}).$

\end{definition}

\begin{prop}\label{PropSigma}
Let $p>2$. Then we have:
\begin{enumerate}
\item[{\rm (i)}]  
$\dim_{\k}\Sigma_{m|n}^{(r)}\;=\; \dim_{\mC}\End_{GL(V_{m|n,\mC})}(V_{m|n,\mC}^{\otimes r})$.
\item[{\rm (ii)}]  There is a commutative diagram of algebra morphisms
$$\xymatrix{
\Sigma_{m|n}^{(r)}\ar@{^{(}->}[rr]^-{\widetilde{\Phi}^{(r)}_{m|n}}&& \End_{GL(m|n)}(V_{m|n}^{\otimes r})\\
\k S_r\ar[u]\ar[urr]_-{\Phi^{(r)}_{m|n}}}$$
where $\widetilde{\Phi}^{(r)}_{m|n}$ is injective.
\item[{\rm (iii)}]  The morphism $\widetilde{\Phi}^{(r)}_{m|n}$ is an isomorphism if and only if 
$$\dim_{\k}\End_{GL(m|n)}(V_{m|n}^{\otimes r})=\dim_{\mC}\End_{GL(m|n)}(V_{m|n,\mC}^{\otimes r}).$$
\end{enumerate}
\end{prop}
\begin{proof}
Part (i) follows from Lemma~\ref{LemApf}(i) and surjectivity of $\mC S_r\to\End_{GL(\mC^{m|n})}(V_{m|n,\mC}^{\otimes r})$.
Part (ii) follows from Lemma~\ref{LemApf}(ii).
Part (iii) follows from the combination of parts (i) and (ii).
\end{proof}

\begin{remark}
Proposition~\ref{PropSigma} reveals a potential approach for super invariant theory in positive characteristic:
\begin{enumerate}
\item Show that $\End_{GL(m|n)}(V_{m|n}^{\otimes r})$ does not depend on the characteristic $p$.
\item Determine explicitly `where to invert powers of $p$' in $\mZ S_r$ (or its relevant quotient corresponding to the characteristic 0 case) before reducing mod $p$, in the creation of $\Sigma_{m|n}^{(r)}.$
\end{enumerate} 
\end{remark}

We demonstrate that, and how, this strategy works in the first case where surjectivity of $\Phi^{(r)}_{m|n}$ fails.
\begin{theorem}\label{ThmSigma}
Consider $p=3$, $r=5$, $m=2$ and $n=1$. 
\begin{enumerate}
\item[{\rm (i)}]  The morphism $\widetilde{\Phi}_{2|1}^{(5)}$ in Proposition~\ref{PropSigma}(ii) is an isomorphism.
\item[{\rm (ii)}]  Let $R$ be the subring of $\mZ[1/3]S_5$ generated by $S_5$ and $a/3$ for $a\in \mZ S_5$ the skew symmetriser, then
$$  \k\otimes R\;\simeq\; \End_{GL(2|1)}(V_{2|1}^{\otimes 5})\;\simeq\; \Sigma^{(5)}_{2|1}.$$
\end{enumerate}
\end{theorem}
\begin{proof}
Part (i) follows immediately from Proposition~\ref{PropSigma}(iii) and Theorem~\ref{Thm21}(ii). Note also that
$$\mZ S_5\;\to\;\End(V_{2|1,\mZ}^{\otimes 5})$$
is injective.

By Theorem~\ref{Thm21}(iii), the kernel of $\Phi_{2|1}^{(5)}$ is 1-dimensional and therefore spanned by the skew symmetriser $a$. It follows easily that in the integral form $a/3$ is sent to an element of $\End(V_{2|1,\mZ}^{\otimes 5})$, whereas $a/9$ is not. The conclusion now follows.
\end{proof}

%
%

\subsection{Characteristic 2}

It is observed in \cite{DF} that the traditional first fundamental theorem for the orthogonal group, for instance formulated as surjectivity of
\begin{equation}\label{BrO}
\cB_r(m)\;\to\; \End_{O(m)}(V_m^{\otimes r}),
\end{equation}
also fails for $p=2$ (it is known to hold for all $p\not=2$, see \cite{DH}).
We show that the procedure of the previous subsection also has potential here.


\subsubsection{}We focus on $O(2)\simeq\mG_m\rtimes \mZ/2$. Denote by $B_2(2)$ the integral Brauer algebra with basis $1,s,e$ (with $e^2=2e$, $se=e=es$ and $s^2=1$).

For $p=2$, denote by $\widetilde{\cB}_2(2)$ the $\k$-algebra obtained by specialising the subring of $\mZ[1/2]\otimes B_2(2)$ generated by $B_2(2)$ and $(1+s-e)/2$. We have a canonical algebra morphism
$$\cB_2(0)=\cB_2(2)\;\to\; \widetilde{\cB}_2(2)$$
and this is precisely what one obtains from applying Construction~\ref{Constr} to \eqref{BrO}.
\begin{prop}\label{Prop2}
For $p=2=m=r$, the algebra morphism \eqref{BrO}
has kernel and cokernel of dimension 1, but extends to an isomorphism
$$ \widetilde{\cB}_2(2)\;\xrightarrow{\sim}\; \End_{O(2)}(V^{\otimes 2}).$$
\end{prop}

\subsection{Double centralisers}

\begin{lemma}
Assume $p\not=2$.
The double centraliser of $\k S_r$ acting on $V_{m|n}^{\otimes r}$ is given by $\End_{GL(m|n)}(V_{m|n}^{\otimes r})$.
\end{lemma}
\begin{proof}
This is a direct consequence of the two equivalent definitions of the Schur superalgebra; one as the centraliser of $\k S_r$ and one as the algebra with represention theory given by polynomial representations of $GL(m|n)$, see \cite{Donkin}.
\end{proof}

\begin{corollary}
Assume $p=3$.
The double centraliser of $\k S_5$ acting on $V_{2|1}^{\otimes 5}$ is the `deformation' of $\k S_5$ described in Theorem~\ref{ThmSigma}(ii).
\end{corollary}

\begin{remark}
 In contrast, the algebra $\End_{O(2)}(V^{\otimes 2}_2)$ from Proposition~\ref{Prop2} is not the double centraliser of $\cB_2(0)$. Indeed, one can compute directly that the double centraliser is just $\k S_2$.
\end{remark}

\end{document}